\documentclass[preprint]{elsarticle}

\usepackage{template}
\usepackage{framed,multirow}

\usepackage{amssymb}
\usepackage{latexsym}
\usepackage{slashbox}

\usepackage{url}
\usepackage{color}
\usepackage{lineno}
\usepackage{amsmath,mathrsfs}
\usepackage{amsthm}
\usepackage[caption=false]{subfig}
\usepackage{dsfont}
\usepackage{float}
\usepackage{gensymb}
\usepackage{natbib}
\usepackage{graphicx,titlesec}
\usepackage{epstopdf}
\usepackage{geometry}
\usepackage{booktabs}
\usepackage{makecell,diagbox}
\usepackage{ulem}

\usepackage{varioref}
\usepackage{hyperref}
\usepackage{cleveref}

\hypersetup{
	colorlinks=true,
	linkcolor={red!50!black},
	citecolor={blue!50!black},
	urlcolor={blue!80!black}
}

\newcommand{\abs}[1]{\left\lvert#1\right\rvert}
\def\norm#1{\|#1\|}

\newtheorem{remark}{Remark}

\newtheorem{prop}{Proposition}

\crefformat{equation}{(#2#1#3)}
\crefmultiformat{equation}{(#2#1#3)}{ and~(#2#1#3)}{, (#2#1#3)}{ and~(#2#1#3)}

\crefformat{figure}{Fig.~#2#1#3}
\crefmultiformat{figure}{Figures~(#2#1#3)}{ and~(#2#1#3)}{, (#2#1#3)}{ and~(#2#1#3)}
\crefformat{table}{Table~#2#1#3}

\graphicspath{{figures/}}

\begin{document}


\begin{frontmatter}

\title{A second-order semi-implicit method for the inertial Landau-Lifshitz-Gilbert equation}

\author[1]{Panchi Li}
\ead{LiPanchi1994@163.com}
\author[2]{Lei Yang}
\ead{leiyang@must.edu.mo}
\author[3]{Jin Lan}
\ead{lanjin@tju.edu.cn}
\author[1,4]{Rui Du\corref{cor1}}
\cortext[cor1]{Corresponding authors}
\ead{durui@suda.edu.cn}
\author[1,4]{Jingrun Chen\corref{cor1}}
\ead{jingrunchen@suda.edu.cn}

\address[1]{School of Mathematical Sciences, Soochow University, Suzhou, 215006, China.}
\address[2]{Faculty of Information Technology, Macau University of Science and Technology, Macao SAR, China.}
\address[3]{Center for Joint Quantum Studies and Department of Physics, School of Science, Tianjin University, 92 Weijin Road, Tianjin 300072, China.}
\address[4]{Mathematical Center for Interdisciplinary Research, Soochow University, Suzhou, 215006, China.}

\begin{abstract}
Electron spins in magnetic materials have preferred orientations collectively and generate the macroscopic magnetization. Its dynamics spans over a wide range of timescales from femtosecond to picosecond, and then to nanosecond. The Landau-Lifshitz-Gilbert (LLG) equation has widely been used in micromagnetics simulations over decades. Recent theoretical and experimental advances show that the inertia of magnetization emerges at sub-picoseconds and contributes to the ultrafast magnetization dynamics which cannot be captured intrinsically by the LLG equation. Therefore, as a generalization, the inertial LLG (iLLG) equation is proposed to model the ultrafast magnetization dynamics. Mathematically, the LLG equation is a nonlinear system of parabolic type with (possible) degeneracy. However, the iLLG equation is a nonlinear system of mixed hyperbolic-parabolic type with degeneracy, and exhibits more complicated structures. It behaves like a hyperbolic system at the sub-picosecond scale while behaves like a parabolic system at larger timescales. Such hybrid behaviors impose additional difficulties on designing numerical methods for the iLLG equation. In this work, we propose a second-order semi-implicit scheme to solve the iLLG equation. The second temporal derivative of magnetization is approximated by the standard centered difference scheme and the first derivative is approximated by the midpoint scheme involving three time steps. The nonlinear terms are treated semi-implicitly using one-sided interpolation with the second-order accuracy. At each step, the unconditionally unique solvability of the unsymmetric linear system of equations in the proposed method is proved with a detailed discussion on the condition number. Numerically, the second-order accuracy in both time and space is verified. Using the proposed method, the inertial effect of ferromagnetics is observed in micromagnetics simulations at small timescales, in consistency with the hyperbolic property of the model at sub-picoseconds. For long time simulations, the results of the iLLG model are in nice agreements with those of the LLG model, in consistency with the parabolic feature of the iLLG model at larger timescales.
\end{abstract}

\begin{keyword}

\KWD Inertial Landau-Lifshitz-Gilbert equation\sep  Semi-implicit scheme\sep Second-order accuracy\sep Micromagnetics simulations\\
\MSC[2000] 35Q99 \sep 65Z05 \sep 65M06
\end{keyword}

\end{frontmatter}


\section{Introduction}

Ferromagnetic materials are widely used for data storage devices thanks to the realization of fast magnetization dynamics under different external controls~\cite{RevModPhys.76.323,Brataas2012}. In this scenario, the dissipative magnetization dynamics is mainly controlled by the slow magnetic degrees of freedom at a timescale of picosecond ($10^{-12}\;$s) to nanosecond ($10^{-9}\;$s), which can be successfully modeled by the classical Landau-Lifshitz-Gilbert equation~\cite{LandauLifshitz1935,Gilbert1955}. Meanwhile, experimentally, ultrafast spin dynamics at the sub-picosecond timescale has been observed~\cite{PhysRevLett1996}, and the magnetization reversal excited by spin wave of sub-$\mathrm{GHz}$ frequency has been realized~\cite{NatComm2011_experiment}, both of which provide new routes for spintronic applications. Theoretically, the inertial term is added to the LLG equation, which accounts for the ultrafast magnetization dynamics and nutation loops~\cite{PhysRevB172403,PhysRevB020410,PhysRevLett057204}.

Introducing $\tau$ as the characteristic timescale of the inertial effect, the magnetization dynamics can be roughly divided into two regimes: the diffusive regime at the timescale of $t\gg\tau$, and the hyperbolic regime at the timescale of $t\approx\tau$. In the hyperbolic regime, magnetization dynamics exhibits the inertial feature~\cite{PhysRevB140413,NatPhys2020}. From the modeling perspective, $\partial_t\mathbf{M}$ and $\mathbf{M}\times \partial_t\mathbf{M}$ control the time evolution of magnetization $\mathbf{M}(\mathbf{x}, t)$ in the LLG equation, and $\partial_{tt}\mathbf{M}$ is further added to account for the inertial effect. This modification leads to the inertial LLG (iLLG) equation \cite{PhysRevB172403,PhysRevB020410}. Mathematically, the LLG equation is a nonlinear system of equations of parabolic type with (possible) degeneracy. Under the condition $t\approx\tau$, the inertial term dominates after nondimensionalization and the iLLG equation is more like a nonlinear system of equations of hyperbolic type. Under the condition $t\gg\tau$, the inertial term can be ignored and the equation is more like a parabolic system. Therefore, a reliable numerical method for the iLLG equation shall capture the inertial dynamics at the sub-picosecond scale and the reversal dynamics at the nanosecond scale.

There exists a large volume of numerical methods for the LLG equation; see \cite{Kruzik:2006,cimrak2007survey} for reviews and references therein. First order semi-implicit schemes include the Gauss-Seidel projection method~\cite{WANG2001357,LI2020109046} and the semi-implicit backward Euler method~\cite{Ivan2005IMA}. The second order semi-implicit projection method with backward differentiation formula has been recently proposed and its second-order accuracy is established theoretically \cite{XIE2020109104,CHEN202155}. Closely related to the current work, the implicit midpoint scheme is available in the literature~\cite{SIAM2006midpoint}. The key idea is to apply the midpoint scheme with three time steps to the temporal derivative $\partial_t\mathbf{M}$ and the centered difference scheme to the second derivative $\partial_{tt}\mathbf{M}$ simultaneously, and then to make nonlinear terms semi-implicit using one-sided interpolation with magnetization at the previous time steps. Implicit schemes are also available in the literature, for example the midpoint scheme~\cite{SIAM2006midpoint}. Theoretically, it is often possible to prove the convergence and energy dissipation of implicit schemes under the assumption that the nonlinear system of equations has a unique solution at each step. Numerically, for larger stepsizes, the convergence of a nonlinear solver such as Netwon's method often slows down. However, multiple solutions may arise in micromagnetics simulations for implicit schemes if the initial guess is chosen arbitrarily. Therefore, semi-implicit schemes are typically superior to implicit schemes~\cite{SUN:2021}.

Numerical methods for the iLLG equation are rarely studied. In~\cite{ruggeri2021numerical}, introducing $\mathbf{V}=\partial_t \mathbf{M}$, $\mathbf{W}=\mathbf{M}\times\partial_t\mathbf{M}$, and using the property $\mathbf{M}\cdot\mathbf{V} = 0$ for any $\mathbf{x}$ and $t$, the author designed the tangent plane scheme~(TPS) with first-order accuracy and angular momentum method~(AMM) with second-order accuracy. In these schemes, $\mathbf{M}$ and $\mathbf{V}$ in TPS, or $\mathbf{M}$ and $\mathbf{W}$ in AMM are treated as unknown fields to be approximated and thus the number of unknowns is doubled. Note that the approximation space for $\mathbf{V}$ has to satisfy $\mathbf{M}\cdot\mathbf{V} = 0$ in a pointwise sense. Therefore, at each step, a nonlinear system of equations has to be solved since the implicit midpoint scheme is applied for the temporal discretization. In this work, we propose a second-order semi-implicit scheme to solve the iLLG equation. The second temporal derivative of magnetization is approximated by the standard centered difference scheme and the first temporal derivative is approximated by the midpoint scheme. The nonlinear terms are treated semi-implicitly using one-sided interpolation with the second-order accuracy. At each temporal step, the unconditionally unique solvability of the unsymmetric linear system of equations in the proposed method is proved with a detailed discussion on the condition number. The unsymmetric linear systems of equations are solved by the GMRES solver~\cite{GMRES}.

The rest of the paper is organized as follows. In \Cref{sec:model}, the iLLG equation is introduced. In \Cref{sec:scheme}, the second-order semi-implicit scheme is proposed, and the unique solvability of the linear system of equations at each step is proved. In \Cref{sec:simulations}, numerically, the second-order accuracy in both space and time is checked, and the dynamics of iLLG equation is studied at different timescales. Conclusions are drawn in \Cref{sec:conclusion}.

\section{The inertial Landau-Lifshitz-Gilbert equation}
\label{sec:model}

The dynamics of magnetization $\mathbf{M}(\mathbf{x}, t)$ in a ferromagnetic medium follows the classical LLG equation~\cite{LandauLifshitz1935,Gilbert1955}
\begin{equation}
  \label{equ:LLG}
  \partial_t\mathbf{M} = -\gamma\mathbf{M}\times\left(\mathbf{H}_{\mathrm{eff}} - \frac{\alpha}{\gamma M_s}\partial_t\mathbf{M}\right)
\end{equation}
with $\gamma$ the gyromagnetic ratio. Below the Curie temperature, $\abs{\mathbf{M}} = M_s$ with $M_s$ the saturation magnetization, which is treated as a constant. On the right-hand side of \eqref{equ:LLG}, the first term is the gyromagnetic term and the second term is the damping term with $\alpha$ being the Gilbert damping parameter. Here, $\mathbf{H}_{\mathrm{eff}} = -\delta\mathcal{F}/\delta\mathbf{M}$ is the effective field, calculated as the variational derivative of the Landau-Lifshitz~(LL) energy $\mathcal{F}$ with respect to magnetization $\mathbf{M}$. The LL energy $\mathcal{F}$ reads as
\begin{equation}
  \label{equ:LLenergy}
  \mathcal{F}[\mathbf{M}] = \frac 1{2}\int_{\Omega}\left[\frac{A}{M_s^2}\abs{\nabla\mathbf{M}}^2 + \Phi\left(\frac{\mathbf{M}}{M_s}\right) - 2\mu_0\mathbf{H}_{\mathrm{e}}\cdot\mathbf{M}\right]\mathrm{d}\mathbf{x} + \frac{\mu_0}{2}\int_{\mathds{R}^3}\abs{\nabla U}^2\mathrm{d}\mathbf{x},
\end{equation}
where $\mu_0$ denotes the permeability of vacuum, $\Omega$ is the volume occupied by the material, $\frac{A}{M_s^2}\abs{\nabla\mathbf{M}}^2$ is the exchange interaction between the neighbors of magnetization, $\Phi\left(\frac{\mathbf{M}}{M_s}\right)$ is the anisotropy energy, $- 2\mu_0\mathbf{H}_{\mathrm{e}}\cdot\mathbf{M}$ is the Zeeman energy resulting from the external magnetic field, and the last term is the dipolar energy resulting from the stray field induced by the magnetization distribution inside the material. The dipolar energy is described by
\[
U = \int_{\Omega}\nabla N(\mathbf{x} - \mathbf{y})\cdot\mathbf{M}(\mathbf{y})\mathrm{d}\mathbf{y},
\]
where $N(\mathbf{x} - \mathbf{y}) = -\frac 1{4\pi}\frac 1{\abs{\mathbf{x} - \mathbf{y}}}$ is the Newtonian potential. Without loss of generality, we assume the material is uniaxial, that is, $\Phi(\mathbf{M}/M_s)=K_u(M_2^2+M_3^2)/M_s^2$ with $K_u$ the anisotropy constant. Then the total effective field reads as
\begin{equation}
  \label{equ:H_effective}
  \mathbf{H}_{\mathrm{eff}} =  \frac{A}{M_s^2}\Delta\mathbf{M} - \frac{K_u}{M_s^2}(M_2\mathbf{e}_2 + M_3\mathbf{e}_2) + \mu_0\mathbf{H}_{\mathrm{e}} + \mu_0\mathbf{H}_{\mathrm{s}}
\end{equation}
with $\mathbf{e}_2 = (0, 1, 0)^T$, $\mathbf{e}_3 = (0, 0, 1)^T$ and $\mathbf{H}_\mathrm{s} = -\nabla U$. An equivalent expression for the anisotropy energy of a uniaxial ferromagnet is $\Phi(\mathbf{M}/M_s) = -K_uM_1^2/M_s^2$, and the corresponding field is $K_uM_1/M_s^2\mathbf{e}_1$ with $\mathbf{e}_1 = (1, 0, 0)^T$.

Recent advances in experiments reveal the ultrafast inertial dynamics~\cite{PhysRevB140413,NatPhys2020} and these observations are successfully interpreted by the inertial effect at the sub-picosecond timescale. To account for the inertial effect, from the modeling perspective, the iLLG equation is proposed as~\cite{PhysRevB172403,PhysRevB020410,PhysRevLett057204}
\begin{equation}
  \label{equ:iLLG}
  \partial_t\mathbf{M} = -\gamma\mathbf{M}\times\left(\mathbf{H}_{\mathrm{eff}} - \frac{\alpha}{\gamma M_s}\left(\partial_t\mathbf{M} + \tau\partial_{tt}\mathbf{M}\right)\right),
\end{equation}
where $\tau$ represents the characteristic timescale of the inertial dynamics of magnetization. Compared with the ferromagnetic counterparts, the dynamics of magnetization in ferrimagnets and antiferromagnets with two sublattics $\mathbf{M}_1$ and
$\mathbf{M}_2$ can be recast to the dynamics of stagger magnetization $\mathbf{n} = (\mathbf{M}_1 - \mathbf{M}_2)/(M_{s1} + M_{s2})$
and the net magnetization $\mathbf{m} = \mathbf{M}_1/M_{s1}+\mathbf{M}_2/M_{s2}$, and the dynamics of $\mathbf{n}$ follows a similar form to the iLLG equation~\cite{PhysRevLett.122.057204,PhysRevB.103.214407}.

\begin{prop}\label{prop:energy}
Define the total energy $\mathcal{J}[\mathbf{M}] = \mathcal{F}[\mathbf{M}] + \frac{\alpha\tau}{2\gamma M_s}\int_{\Omega}\abs{\partial_t\mathbf{M}}^2\mathrm{d}\mathbf{x}$. Then the following energy law holds
\begin{equation}\label{eqn:ene dissip1}
\frac{\mathrm{d}\mathcal{J}[\mathbf{M}]}{\mathrm{d}t} = -\frac{\alpha}{\gamma M_s}\int_{\Omega}\abs{\partial_t\mathbf{M}}^2\mathrm{d}\mathbf{x} + \mu_0\int_{\Omega}\mathbf{M}\cdot\mathbf{H}'_{\mathrm{e}}(t)\mathrm{d}\mathbf{x}.
\end{equation}
For the zero or constant external magnetic field, the following energy dissipation relation holds
\begin{equation}\label{eqn:ene dissip2}
\frac{d\mathcal{J}[\mathbf{M}]}{dt} = -\frac{\alpha}{\gamma M_s}\int_{\Omega}\abs{\partial_t\mathbf{M}}^2\mathrm{d}\mathbf{x}\leq 0.
\end{equation}
\end{prop}
\begin{proof}
The inner product between \eqref{equ:iLLG} and $\mathbf{H}_{\mathrm{eff}} - \frac{\alpha}{\gamma M_s}\left(\partial_t\mathbf{M} + \tau\partial_{tt}\mathbf{M}\right)$ yields
\begin{equation}\label{eqn:identity}
  \int_{\Omega}\partial_t\mathbf{M}\cdot\left(\mathbf{H}_{\mathrm{eff}} - \frac{\alpha}{\gamma M_s}\left(\partial_t\mathbf{M} + \tau\partial_{tt}\mathbf{M}\right)\right)\mathrm{d}\mathbf{x} = 0.
\end{equation}
Therefore,
\begin{align*}
  \frac{d\mathcal{J}[\mathbf{M}]}{dt} &= \int_{\Omega} \left\{\left(\partial_t\mathbf{M} \cdot(-\mathbf{H}_{\mathrm{eff}}) + \mu_0\mathbf{M}\cdot\mathbf{H}'_{\mathrm{e}}(t) \right)  + \frac{\alpha\tau}{\gamma M_s}\partial_t\mathbf{M}\cdot\partial_{tt}\mathbf{M}\right\}\mathrm{d}\mathbf{x}\\
  &\overset{\eqref{eqn:identity}}{=} \int_{\Omega} \left\{-\frac{\alpha}{\gamma M_s}\abs{\partial_t\mathbf{M}}^2 + \mu_0\mathbf{M}\cdot\mathbf{H}'_{\mathrm{e}}(t) \right\}\mathrm{d}\mathbf{x},
\end{align*}
which completes the derivation of \eqref{eqn:ene dissip1}. If $\mathbf{H}'_{\mathrm{e}}(t)$ vanishes, \eqref{eqn:ene dissip1} reduces to \eqref{eqn:ene dissip2}.
\end{proof}

\begin{remark}
The definition of total energy is well-defined once the following initial condition is introduced
  \begin{equation}
    \partial_t\mathbf{M}(\mathbf{x},0) = 0,
  \end{equation}
which provides the other initial condition for the second-order iLLG equation. Note that $\mathcal{J}[\mathbf{M}] = \mathcal{F}[\mathbf{M}]$ at $t = 0$.
  \label{rk:initial-condition}
\end{remark}

To ease the description, we rewrite the iLLG equation \eqref{equ:iLLG} in a dimensionless by defining $\mathbf{H}_{\mathrm{eff}} = \mu_0M_s\mathbf{h}$, $\mathbf{H}_{\mathrm{e}} =M_s\mathbf{h}_{\mathrm{e}}$, $\mathbf{H}_{\mathrm{s}} = M_s\mathbf{h}_{\mathrm{s}}$, and $\mathbf{M} = M_s\mathbf{m}$. After the spatial rescaling $x\rightarrow Lx$ (still use $x$ after rescaling) with $L$ being the diameter of $\Omega$, the dimensionless form of LL energy functional is
\begin{equation}
\label{equ:LLenergydimensionless}
\tilde{\mathcal{F}}[\mathbf{m}] = \frac{1}{2}\int_{\Omega'}\left[\epsilon\abs{\nabla\mathbf{m}}^2 + q(m_2^2 + m_3^2) - 2\mathbf{h}_{\mathrm{e}}\cdot\mathbf{m}\right]\mathrm{d}\mathbf{x} - \int_{\Omega'}\mathbf{h}_{\mathrm{s}}\cdot\mathbf{m}\mathrm{d}\mathbf{x},
\end{equation}
where $\epsilon=A/(\mu_0M_s^2L^2)$, $q=K_u/(\mu_0M_s^2)$ and $\tilde{\mathcal{F}}[\mathbf{m}]=\mathcal{F}[\mathbf{M}]/(\mu_0M_s^2)$. Meanwhile, after the temporal rescaling $t\rightarrow t(M_s\mu_0\gamma)^{-1}$ (still use $t$ after rescaling), the dimensionless form of \eqref{equ:iLLG} reads as
\begin{equation}
\label{equ:dimensionlessILLG}
\partial_t\mathbf{m} = -\mathbf{m}\times\mathbf{h} + \alpha\mathbf{m}\times\left( \partial_t\mathbf{m} + \eta\partial_{tt}\mathbf{m} \right)
\end{equation}
with
\begin{equation}
  \label{equ:dimensionlesseffectivefield}
  \mathbf{h} = \epsilon\Delta\mathbf{m} - q(m_2\mathbf{e}_2 + m_3\mathbf{e}_3) + \mathbf{h}_e + \mathbf{h}_s.
\end{equation}
The dimensionless parameter $\eta = \tau/(\mu_0\gamma M_s)^{-1}$, represents the ratio between the characteristic timescale of inertial dynamics and the characteristic timescale of magnetization dynamics. Homogeneous Neumann boundary condition is used
\begin{equation}
  \label{equ:Neumann}
  \frac{\partial\mathbf{m}}{\partial\boldsymbol{\nu}}\Big\vert_{\partial\Omega} = 0,
\end{equation}
where $\boldsymbol{\nu}$ represents the outward unit normal vector along $\partial\Omega$.

\section{A second-order semi-implicit finite difference scheme}
\label{sec:scheme}

Denote
\begin{equation}
\label{equ:ideal_effective_field}
\mathbf{f}(\mathbf{m}^n) = - q(m^n_2\mathbf{e}_2 + m^n_3\mathbf{e}_3) + \mathbf{h}_e + \mathbf{h}^n_s.
\end{equation}
For \cref{equ:dimensionlessILLG}-\eqref{equ:Neumann}, we employ a midpoint scheme with three time steps
\begin{multline}
  \label{equ:implicit_midpoints}
  \frac{\mathbf{m}^{n+1} - \mathbf{m}^{n-1}}{2\Delta t} = -\frac{\mathbf{m}^{n+1} + \mathbf{m}^{n-1}}{2}\times\left(\epsilon\Delta_h\frac{\mathbf{m}^{n+1} + \mathbf{m}^{n-1}}{2} + \mathbf{f}(\frac{\mathbf{m}^{n+1} + \mathbf{m}^{n-1}}{2})\right) \\
+\alpha\frac{\mathbf{m}^{n+1} + \mathbf{m}^{n-1}}{2}\times\left(\frac{\mathbf{m}^{n+1} - \mathbf{m}^{n-1}}{2\Delta t} + \eta\frac{\mathbf{m}^{n+1} -2\mathbf{m}^{n} + \mathbf{m}^{n-1}}{\Delta t^2}\right),
\end{multline}
where $\Delta_h$ represents the standard second-order centered difference stencil. For a 3D Cartesian mesh with indices $j = 0,1,\cdots, nx,nx+1$, $k = 0,1,\cdots, ny,ny+1$ and $l = 0,1,\cdots, nz, nz+1$. The second-order centered difference for $\Delta_h \mathbf{m}_{j,k,l}$ reads as
\begin{equation}
\begin{aligned}
\Delta_h\mathbf{m}_{j,k,l} = &\frac{\mathbf{m}_{j+1,k,l}-
	2\mathbf{m}_{j,k,l}+\mathbf{m}_{j-1,k,l}}{\Delta x^2} \\
&+\frac{\mathbf{m}_{j,k+1,l}-2\mathbf{m}_{j,k,l}+
	\mathbf{m}_{j,k-1,l}}{\Delta y^2} \\
&+\frac{\mathbf{m}_{j,k,l+1}-2\mathbf{m}_{j,k,l}+
	\mathbf{m}_{j,k,l-1}}{\Delta z^2},
\end{aligned}
\label{LaplacianDiscrete}
\end{equation}
where $\mathbf{m}_{j,k,l}=\mathbf{m}((j - \frac 1{2})\Delta x,
(k - \frac 1{2})\Delta y, (l - \frac 1{2})\Delta z)$.
For the Neumann boundary condition, a second-order approximation yields
\begin{align*}
\mathbf{m}_{0,k,l} & =\mathbf{m}_{1,k,l},\quad \mathbf{m}_{nx,k,l}  =\mathbf{m}_{nx+1,k,l},\quad k = 1,\cdots,ny,l=1,\cdots,nz, \\
\mathbf{m}_{j,0,l} & =\mathbf{m}_{j,1,l},\quad \mathbf{m}_{j,ny,l}  =\mathbf{m}_{j,ny+1,l},\quad j = 1,\cdots,nx,l=1,\cdots,nz, \\
\mathbf{m}_{j,k,0} & =\mathbf{m}_{j,k,1},\quad \mathbf{m}_{j,k,nz}  =\mathbf{m}_{j,k,nz+1},\quad j = 1,\cdots,nx,k=1,\cdots,ny.
\end{align*}
\eqref{equ:implicit_midpoints} is an implicit scheme with the truncation error $\mathcal{O}(h^2 + \Delta t^2)$ with $\Delta x = \Delta y = \Delta z = h$.
At each step, a nonlinear system of equations has to be solved. In addition, for micromagnetics simulations, evaluation of the stray field $\mathbf{h}_{s}^{n+1}$ is also computationally expensive.

To overcome these issues, we propose the following second-order semi-implicit scheme
\begin{equation}
  \label{equ:2-semi-implicit-illg}
  \left\{\begin{aligned}
  \frac{\mathbf{\tilde{m}}^{n+1} - \mathbf{m}^{n-1}}{2\Delta t} = &-\mathbf{m}^n\times\left(\epsilon\Delta_h\left(\frac{\mathbf{\tilde{m}}^{n+1} + \mathbf{m}^{n-1}}{2}\right) + \mathbf{f}(\mathbf{m}^n)\right) +\\ &\alpha\mathbf{m}^n\times\left(\frac{\mathbf{\tilde{m}}^{n+1}-\mathbf{m}^{n-1}}{2\Delta t} + \eta\frac{\mathbf{\tilde{m}}^{n+1}-2\mathbf{m}^n+\mathbf{m}^{n-1}}{\Delta t^2}\right).\\
\mathbf{m}^{n+1} = &\frac 1{\abs{\mathbf{\tilde{m}}}^{n+1}}\mathbf{\tilde{m}}^{n+1}.
\end{aligned}\right.
\end{equation}
By rewriting \eqref{equ:2-semi-implicit-illg}, at each step, only a linear system of equations with the unsymmetric structure needs to be solved
\begin{multline}\label{equ:2-semi-implicit-illglinear}
\left(I + \epsilon\Delta t\mathbf{m}^n\times\Delta_h - \alpha\left(1+\frac{2\eta}{\Delta t}\right)\mathbf{m}^n\times\right)\mathbf{\tilde{m}}^{n+1} = \mathbf{m}^{n-1} - \epsilon\Delta t\mathbf{m}^n\times\Delta_h\mathbf{m}^{n-1} - \\
\alpha\left(1-\frac{2\eta}{\Delta t}\right)\mathbf{m}^n\times\mathbf{m}^{n-1} - 2\Delta t\mathbf{m}^n\times\mathbf{f}(\mathbf{m}^n).
\end{multline}
For micromagnetics simulations, to update $\mathbf{m}^{n+1}$, only $\mathbf{h}_{s}^{n}$ is needed.

Following \cite{XIE2020109104}, we establish the unique solvability of the proposed scheme \eqref{equ:2-semi-implicit-illg} as follows. First, the unique solvability of \eqref{equ:2-semi-implicit-illglinear} is given by the following proposition.
\begin{prop}
  \label{lem:unique-solvability}
Denote $S = \epsilon\Delta t\Delta_h - \alpha\left(1+\frac{2\eta}{\Delta t}\right) I$, $I$ the identity operator, and $A = \mathbf{m}^n\times$ an antisymmetric matrix. Given $\mathbf{m}^{n-1}$ and $\mathbf{m}^n$, then $\det{(I + A S)} \neq 0$, that is, \eqref{equ:2-semi-implicit-illglinear} admits a unique solution for any positive $h$ and $\Delta t$.
\end{prop}
\begin{proof}
Since $\Delta_h$ is a symmetric positive definite matrix, there exists a nonsingular matrix such that $\Delta_h = C^TC$. Therefore, for the coefficient matrix $I + AS$ in \eqref{equ:2-semi-implicit-illglinear}, we have
  \begin{align*}
    \det{(I + AS)} = \det{(I + AC^TC)} = \det{(I + CAC^T)},
  \end{align*}
where $(CAC^T)^T = -CAC^T$ is an antisymmetric matrix. Thus eigenvalues of $CAC^T$ are either $0$ or pure imaginary, $\det{(I + AS)} \neq 0$, which is independent of $h$ and $\Delta t$.
\end{proof}
Second, in \eqref{equ:2-semi-implicit-illg}, a projection step is applied after solving \eqref{equ:2-semi-implicit-illglinear}. The following proposition guarantees that the denominator is always nonzero.
\begin{prop}
\label{lem:unique-solvability2}
If $\mathbf{m}^{0}\cdot\mathbf{m}^{1}\neq 0$ in the pointwise sense, then $\abs{\tilde{\mathbf{m}}^{n}} \neq 0$ at any step $n$.
\end{prop}
\begin{proof}
Multiplying the first equation in \eqref{equ:2-semi-implicit-illg} by $\mathbf{m}^n$ produces
\[
\mathbf{\tilde{m}}^{n+1}\cdot \mathbf{m}^n = \mathbf{m}^n\cdot\mathbf{m}^{n-1}.
\]
When $n=1$, we have $\mathbf{\tilde{m}}^{2}\cdot \mathbf{m}^1 = \mathbf{m}^1\cdot\mathbf{m}^{0} \neq 0$, which implies $\abs{\mathbf{\tilde{m}}^{2}}\neq0$. When $n=2$, we have $\mathbf{\tilde{m}}^{3}\cdot \mathbf{m}^2 = \mathbf{m}^2\cdot\mathbf{m}^{1} = \frac1{\abs{\mathbf{\tilde{m}}^{2}}} \mathbf{m}^1\cdot\mathbf{m}^{0} \neq 0$, which implies $\abs{\mathbf{\tilde{m}}^{3}}\neq0$. Repeating this process completes the proof.
\end{proof}

\begin{remark}
After time rescaling $t\rightarrow(1+\alpha^2)t$, \eqref{equ:dimensionlessILLG} is equivalent to the LL form
  \begin{equation}
    \partial_t\mathbf{m} = -\mathbf{m}\times\left(\mathbf{h} - \tilde\eta \partial_{tt}\mathbf{m} \right) - \alpha\mathbf{m}\times\left(\mathbf{m}\times\left(\mathbf{h} - \tilde\eta \partial_{tt}\mathbf{m} \right)\right),
  \end{equation}
  where $\tilde\eta = \frac{\alpha\eta}{(1+\alpha^2)^2}$. Then the corresponding second-order semi-implicit scheme reads as
  \begin{equation}
    \label{equ:2-semi-implicit-illg-equivalent}
    \left\{\begin{aligned}
    \frac{\mathbf{\tilde{m}}^{n+1} - \mathbf{m}^{n-1}}{2\Delta t} &= -\mathbf{m}^n\times\left(\epsilon\Delta_h\left(\frac{\mathbf{\tilde{m}}^{n+1} + \mathbf{m}^{n-1}}{2}\right) - \tilde\eta\frac{\mathbf{\tilde{m}}^{n+1}-2\mathbf{m}^n+\mathbf{m}^{n-1}}{\Delta t^2} + \mathbf{f}(\mathbf{m}^n)\right) -\\ &\alpha\mathbf{m}^n\times\left(\mathbf{m}^n\times\left(\epsilon\Delta_h\left(\frac{\mathbf{\tilde{m}}^{n+1} + \mathbf{m}^{n-1}}{2}\right) - \tilde\eta\frac{\mathbf{\tilde{m}}^{n+1}-2\mathbf{m}^n+\mathbf{m}^{n-1}}{\Delta t^2} + \mathbf{f}(\mathbf{m}^n)\right)\right),\\
    \mathbf{m}^{n+1} &= \frac 1{\abs{\mathbf{\tilde{m}}^{n+1}}}\mathbf{\tilde{m}}^{n+1}.
    \end{aligned}\right.
  \end{equation}
\end{remark}

Next, we discuss how to solve the linear systems of equations in \eqref{equ:2-semi-implicit-illglinear}. Due to the unsymmetric structure, the GMRES solver is employed. Since the convergence of GMRES depends on the condition number of the linear system, we provide a heuristic demonstration on how the condition number depends on the damping parameter $\alpha$ and the characteristic timescale of the inertial effect $\eta$.

For simplicity, consider $\mathbf{m}^n = \mathbf{e}_1$ and let $S = \epsilon\Delta t\Delta_h - \alpha(1+\frac{2\eta}{\Delta t})I$. In 1D, for the homogeneous Neumann boundary condition, the eigenvalues of $\Delta_x$ are $\frac{-4}{\Delta x^2}\sin^2(\frac{j\pi\Delta x}{2})$~$(j=1,\cdots, nx)$, and the eigenvalues of $S$ are $\lambda_j = -\frac{4\epsilon\Delta t}{\Delta x^2}\sin^2(\frac{j\pi\Delta x}{2}) - \alpha(1+\frac{2\eta}{\Delta t})$, respectively.
Therefore, the $3nx$ eigenvalues of the antisymmetric matrix
\begin{equation*}
  \mathbf{e}_1\times S = \begin{pmatrix}
    0 & 0 & 0\\
    0 & 0 & -S\\
    0 & S & 0
  \end{pmatrix}
\end{equation*}
are
\begin{equation*}
  \underbrace{0, \cdots, 0}_{nx}, \underbrace{\pm\lambda_1i, \cdots, \pm\lambda_{nx}i}_{nx}.
\end{equation*}
Consequently, eigenvalues of $\left(I + \epsilon\Delta t\mathbf{m}^n\times\Delta_x - \alpha\left(1+\frac{2\eta}{\Delta t}\right)\mathbf{m}^n\times\right)$ are
\begin{equation*}
  \underbrace{1, \cdots, 1}_{nx}, 1\pm\lambda_1i, \cdots, 1\pm\lambda_{nx}i.
\end{equation*}
The condition number of $\left(I + \epsilon\Delta t\mathbf{m}^n\times\Delta_x - \alpha\left(1+\frac{2\eta}{\Delta t}\right)\mathbf{m}^n\times\right)$ is
\begin{equation}
  \label{equ:condition-number1D}
  \kappa = \sqrt{1 + \left(\frac{4\epsilon\Delta t}{\Delta x^2} + \alpha\left(1+\frac{2\eta}{\Delta t}\right)\right)^2}.
\end{equation}
Similarly, in 3D, the eigenvalues of $\Delta_h$ are
$$
\lambda_{jkl} = -\frac{4}{\Delta x^2}\sin^2\left(\frac{j\pi\Delta x}{2}\right) -\frac{4}{\Delta y^2}\sin^2\left(\frac{k\pi\Delta y}{2}\right) -\frac{4}{\Delta z^2}\sin^2\left(\frac{l\pi\Delta z}{2}\right),
$$
and the condition number of $\left(I + \epsilon\Delta t\mathbf{m}^n\times\Delta_h - \alpha\left(1+\frac{2\eta}{\Delta t}\right)\mathbf{m}^n\times\right)$ is
\begin{equation}\label{equ:condition-number3D}
\kappa = \sqrt{1 + \left[4\epsilon\Delta t\left(\frac{1}{\Delta x^2} + \frac{1}{\Delta y^2} + \frac{1}{\Delta z^2}\right) + \alpha\left(1+\frac{2\eta}{\Delta t}\right)\right]^2}.
\end{equation}
According to the convergence theory of GMRES \cite{GMRES1986}, the number of iterations is proportional to the condition number \eqref{equ:condition-number1D} in 1D or \eqref{equ:condition-number3D} in 3D, respectively. For smaller $\alpha$ and $\tau$, less number of iterations is needed in GMRES.

Numerically, the number of iterations in GMRES for 3D micromagnetics simulations is recorded for different damping parameters $\alpha$ and different characteristic timescales $\tau$ in \cref{tab:number-iterations}. Material parameters can be found in \Cref{subsec:simulations}. It is easy to find that the number of iterations in GMRES reduces for a smaller damping parameter and a smaller characteristic parameter, which is consistent with our heuristic derivation.
\begin{table}[htbp]
  \centering
  \caption{Number of iterations in GMRES at a single step with $\mathbf{m}^{0} = \mathbf{m}^1 = \mathbf{e}_1$ and $\Delta t = 0.1\mathrm{ps}$ in micromagnetics simulations. The stopping tolerance in GMRES is set to be $1.0e-11$.}
  \begin{tabular}{||c|c|c||}
    \hline
    \diagbox{Damping}{Number of iterations}{Inertial} & $\tau = 1.0\times10^{-13}\mathrm{s}$ & $\tau = 1.0\times10^{-11}\mathrm{s}$ \\
    \hline
    $\alpha = 0.1$ & 9 & 9 \\
    \hline
    $\alpha = 0.01$ &7 & 11 \\
    \hline
    $\alpha = 0.001$ & 6 & 9 \\
    \hline
  \end{tabular}
  \label{tab:number-iterations}
\end{table}

\section{Numerical results}
\label{sec:simulations}

We first verify the second-order accuracy in both time and space in 1D and 3D. For comparison, the full iLLG equation is simplified as
\begin{equation}
  \label{equ:simplicity-equation}
  \partial_t\mathbf{m} = -\mathbf{m}\times\Delta\mathbf{m} + \alpha\left(\partial_t\mathbf{m} + \eta \partial_{tt}\mathbf{m}\right) + \mathbf{g},
\end{equation}
where $\mathbf{g}$ is the source term. The $L^{\infty}$ error $\norm{\mathbf{m}_{\mathrm{e}} - \mathbf{m}_{\mathrm{h}}}_{\infty}$ is recorded with $\mathbf{m}_{\mathrm{e}}$ being the exact solution and $\mathbf{m}_{\mathrm{h}}$ being the numerical solution, respectively.

\subsection{Accuracy check}
\label{subsec:accuracy-check}

Consider the exact solution in 1D for \eqref{equ:simplicity-equation}
\begin{equation*}
\mathbf{m}_{\mathrm{e}} = (\cos(\bar{x})\sin(t), \sin(\bar{x})\sin(t), \cos(t))^T
\end{equation*}
with $\bar{x}=x^2(1-x)^2$. The source term is $\mathbf{g} = \partial_t\mathbf{m}_{\mathrm{e}} + \mathbf{m}_{\mathrm{e}}\times\partial_{xx}\mathbf{m}_{\mathrm{e}} - \alpha\left(\partial_t\mathbf{m}_{\mathrm{e}} + \eta\partial_{tt}\mathbf{m}_{\mathrm{e}}\right)$. The final time $T = 0.5$. As shown in \cref{tab:1daccuracy}, the second-order accuracy is obtained in both time and space.
\begin{table}[htbp]
	\centering
    \caption{The $L^\infty$ error in terms of the temporal stepsize and the spatial gridsize in 1D. Temporal accuracy check: $\Delta x = 1.d-03$; Spatial accuracy check: $\Delta t = 5.0e-03$. $nx = 1/\Delta x$ and $nt = T/\Delta t$.}
	\begin{tabular}{||c|c|ccccc|c||}
		\hline
		\multirow{4}{*}{\makecell[c]{$\alpha = 0.0$\\$\eta=0.0$}} & \multirow{2}{*}{Space} & $nx$ & 20 & 40 & 80 & 160 & order \\
		& & error & 2.74e-4 & 6.98e-05 & 1.88e-05 & 6.07e-06 & 1.84\\
		\cline{2-8}
		& \multirow{2}{*}{Time} & $nt$ & 20 & 40 & 80 & 160 & order \\
		& & error & 4.56e-05 & 1.15e-05 & 2.96e-06 & 8.23e-07 & 1.93\\
		\hline
		\multirow{4}{*}{\makecell[c]{$\alpha = 0.01$\\$\eta=0.0$}} & \multirow{2}{*}{Space} & $nx$ & 20 & 40 & 80 & 160 & order \\
		& & error & 2.73e-04 & 6.97e-05 & 1.88e-05 & 6.07e-06 & 1.84\\
		\cline{2-8}
		& \multirow{2}{*}{Time} & $nt$ & 20 & 40 & 80 & 160 & order \\
		& & error & 4.56e-05 & 1.15e-05 & 2.96e-06 & 8.23e-07 & 1.93\\
		\hline
		\multirow{4}{*}{\makecell[c]{$\alpha = 0.01$\\$\eta=100.0$}} & \multirow{2}{*}{Space} & $nx$ & 20 & 40 & 80 & 160 & order \\
		& & error & 9.95e-05 & 2.56e-05 & 6.92e-06 & 2.24e-06 & 1.83\\
		\cline{2-8}
		& \multirow{2}{*}{Time} & $nt$ & 20 & 40 & 80 & 160 & order \\
		& & error & 1.63e-05 & 4.19e-06 & 1.09e-06 & 3.03e-07 & 1.92\\
		\hline
		\multirow{4}{*}{\makecell[c]{$\alpha = 0.01$\\$\eta=1000.0$}} & \multirow{2}{*}{Space} & $nx$ & 20 & 40 & 80 & 160 & order \\
		& & error & 2.39e-05 & 7.62e-06 & 2.12e-06 & 5.49e-07 & 1.82\\
		\cline{2-8}
		& \multirow{2}{*}{Time} & $nt$ & 20 & 40 & 80 & 160 & order \\
		& & error & 1.83e-06 & 4.68e-07 & 1.21e-07 & 3.34e-08 & 1.93\\
		\hline
	\end{tabular}
	\label{tab:1daccuracy}
\end{table}

Consider the 3D exact solution
\begin{equation*}
\mathbf{m}_{\mathrm{e}} = (\cos(\bar{x}\bar{y}\bar{z})\sin(t), \sin(\bar{x}\bar{y}\bar{z})\sin(t), \cos(t))^T
\end{equation*}
with $\bar{y}=y^2(1-y)^2$ and $\bar{z}=z^2(1-z)^2$. The source term is $\mathbf{g} = \partial_t\mathbf{m}_{\mathrm{e}} + \mathbf{m}_{\mathrm{e}}\times\Delta\mathbf{m}_{\mathrm{e}} - \alpha\left(\partial_t\mathbf{m}_{\mathrm{e}} + \eta\partial_{tt}\mathbf{m}_{\mathrm{e}}\right)$. For the temporal accuracy check, $\Omega = [0, 0.01]^3$ is discretized uniformly with $10$ mesh grids along each direction and $T=0.5$. For the spatial accuracy check, $\Omega=[0, 1]^3$ is discretized uniformly with $6,\,8,\,10,\,12$ grid points along each direction, respectively, and $T = 0.1$ with $\Delta t = 1.0e-03$. The damping parameter $\alpha$ and the inertial parameter $\eta$ are $0.01$ and $1000$, respectively. \cref{fig:3daccuracy} plots the $L^\infty$ error in terms of the temporal stepsize and the spatial gridsize in 3D, respectively. The second-order accuracy is also obtained in both time and space.
\begin{figure}[htbp]
	\centering
	\subfloat[Temporal accuracy]{\includegraphics[width=2.5in]{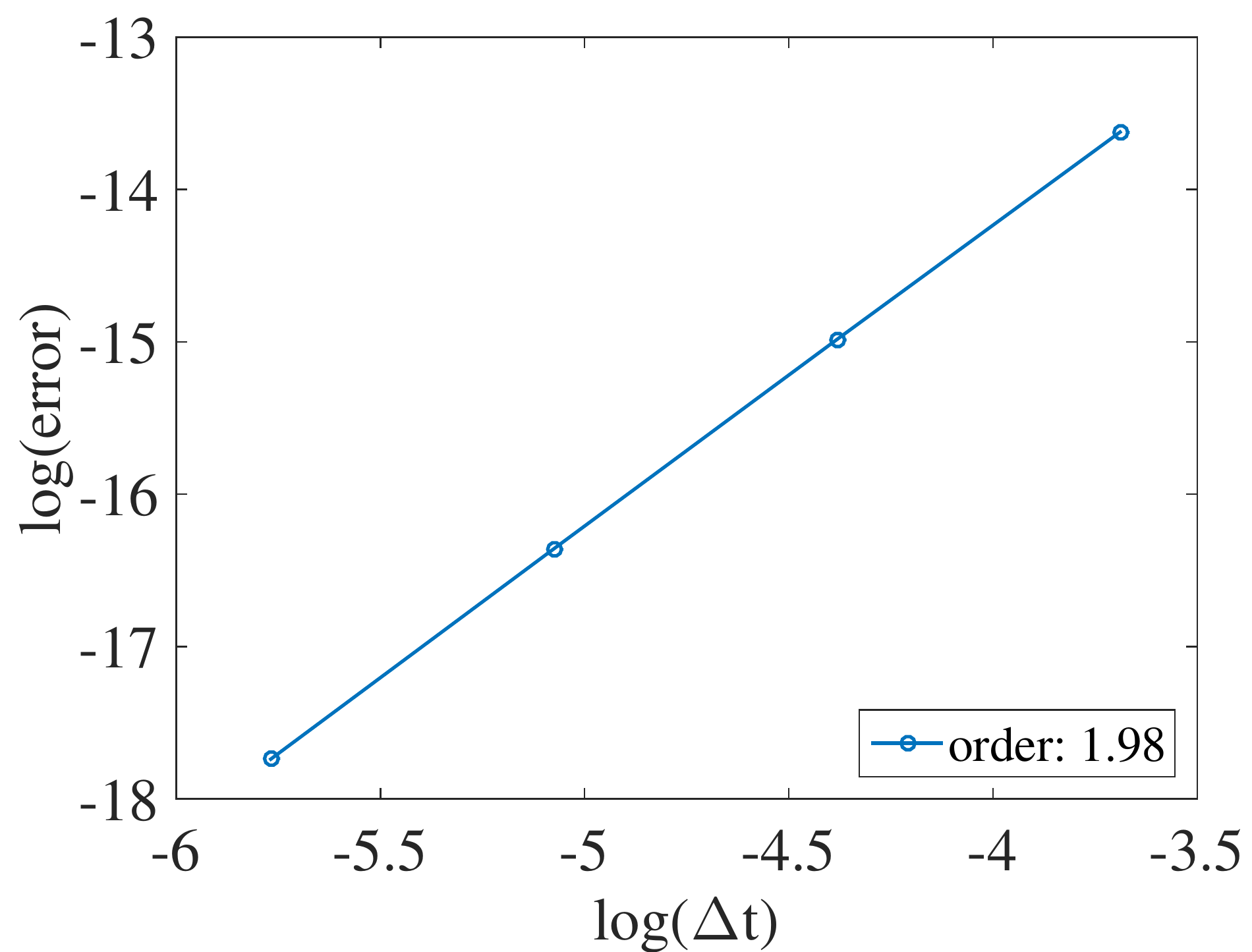}}
	\subfloat[Spatial accuracy]{\includegraphics[width=2.5in]{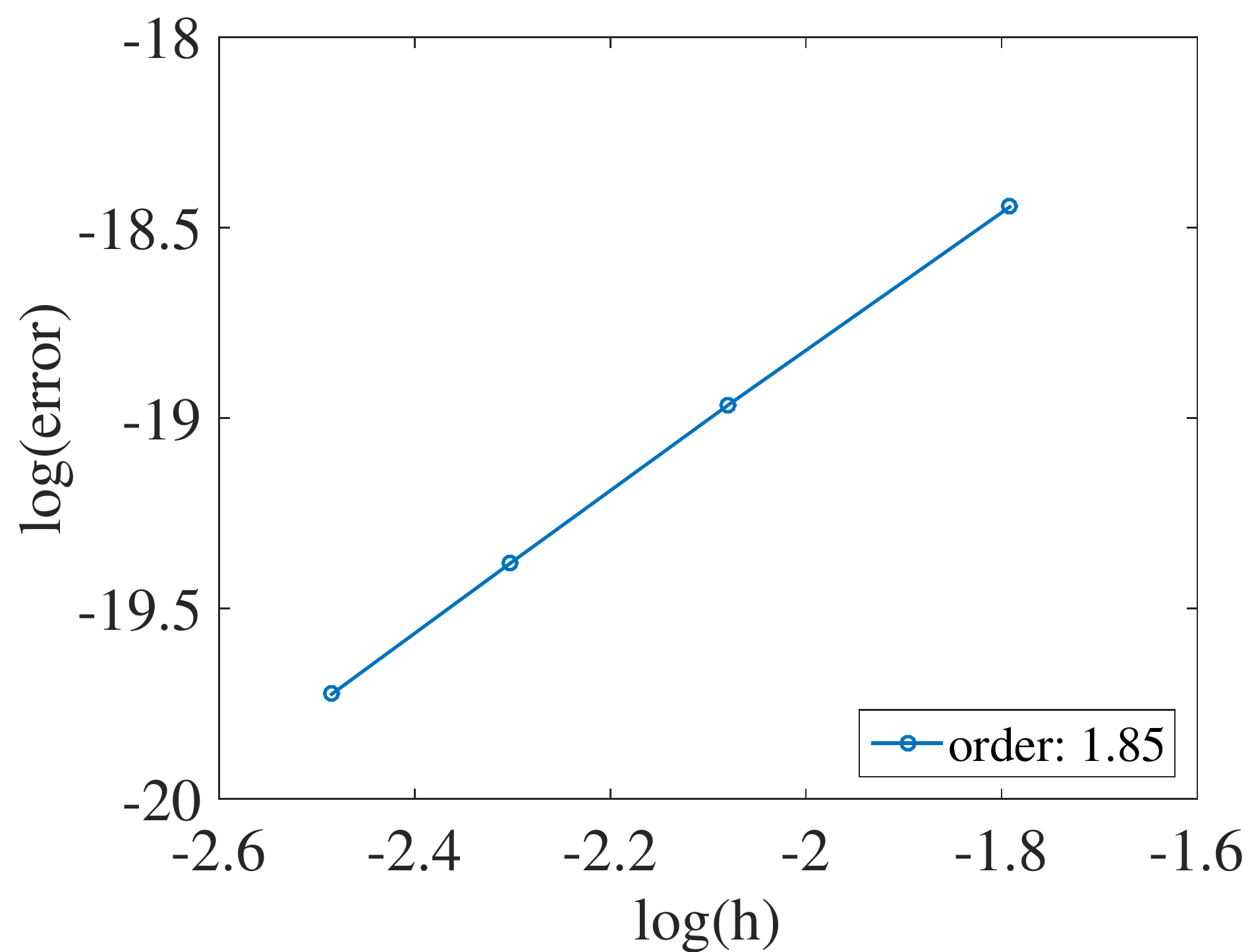}}
	\caption{The $L^\infty$ error in terms of the temporal stepsize and the spatial gridsize in 3D. The second-order accuracy is obtained in both time and space.}
	\label{fig:3daccuracy}
\end{figure}

\subsection{Micromagnetics simulations}
\label{subsec:simulations}

In this part, we simulate the ultrafast magnetization dynamics using the full iLLG equation to study the inertial effect of ferromagnets. Due to the lack of the exact solution, the proposed method is initialized with $\mathbf{m}^{1} = \mathbf{m}^{0}$, in consistency with the initial condition proposed in \Cref{rk:initial-condition}. The setup is physically sound since the ferromagnetic material is initially under equilibrium and no external field is applied. Since the stray field term typically dominates the LL energy in micromagnetics simulations, we first neglect it to check the importance of the inertial term. After that, we turn on the stray field and study the magnetization dynamics under an external magnetic field with high frequency.

In the first test, parameters of the ferromagnetic sample are: $200\mathrm{nm}\times100\mathrm{nm}\times5\mathrm{nm}$ with the mesh size $4\mathrm{nm}\times4\mathrm{nm}\times5\mathrm{nm}$, $\alpha = 0.02$, $M_s = 8.0\times10^5\mathrm{A}/\mathrm{m}$, $A = 1.3\times10^{-11}\mathrm{J}/\mathrm{m}$ and $K_u = 5.0\times10^2\mathrm{J}/\mathrm{m}^3$~\cite{NatPhys2020}. When the inertial parameter $\tau$ is set to be $1.0\times10^{-12}\mathrm{s}$, $\epsilon\sim\mathcal{O}(10^{-4})$, $q\sim\mathcal{O}(10^{-3})$ and $\eta\sim\mathcal{O}(10^{-1})$ after nondimensionalization. This explains why $\eta$ is set to be $1000$ in both \cref{tab:1daccuracy} and \cref{fig:3daccuracy} when $\epsilon$ is set to be $1$. The initial magnetization is chosen as the uniform state $\mathbf{m}^{0} = \mathbf{e}_1$ and a magnetic pulse with high frequency $\mathbf{H}_{\mathrm{e}} = \mu_0M_sF(t)\mathbf{e}_2$ is applied along the $\mathbf{e}_2$ direction, where $F(t) = 0.01\sin(2\pi ft)\chi_{0\leq t\leq 2.0\times10^{-12}}$ and $f = 500\mathrm{GHz}$~\cite{ruggeri2021numerical}. \cref{fig:LLG.vs.iLLG} plots spatially averaged magnetization profiles as functions of time when $\tau = 1.0\times10^{-10}\mathrm{s}$ and $T = 100\mathrm{ps}$ with $\Delta t = 10\mathrm{fs}$. Distinguishable differences are found between the results of LLG equation and iLLG equation. Oscillatory magnetization profiles along $\mathbf{e}_2$ and $\mathbf{e}_3$ directions are observed if the inertial term is present, which is known as the inertial dynamics. \cref{fig:energy_nostray} plots the LL energy (without the stray field term) as a function of time when a magnetic pulse is applied for the LLG equation and the iLLG equation. Results of the iLLG equation show a more profound response to the applied magnetic pulse, and thus the inertial dynamics is generated at sub-picosecond and picosecond timescales.
\begin{figure}[htbp]
  \centering
  \includegraphics[width=6in]{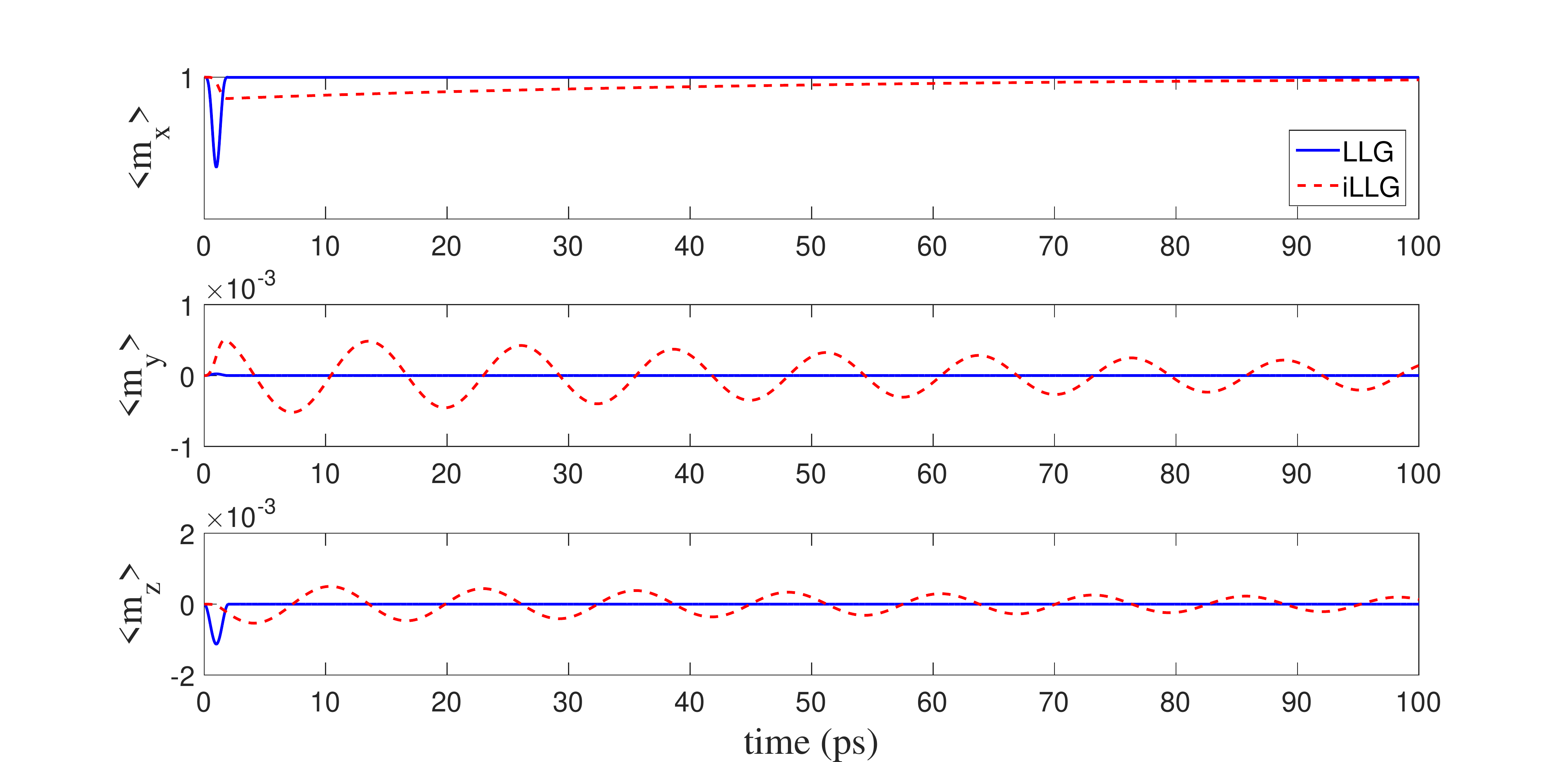}
  \caption{Spatially averaged magnetization profiles as functions of time when $\tau = 1.0\times10^{-10}\mathrm{s}$ and $T = 100\mathrm{ps}$ with $\Delta t = 10\mathrm{fs}$. Distinguishable differences are found between the results of LLG equation and iLLG equation. Oscillatory magnetization profiles (inertial dynamics) along $\mathbf{e}_2$ and $\mathbf{e}_3$ directions are observed if the inertial term is present.}
  \label{fig:LLG.vs.iLLG}
\end{figure}
\begin{figure}[htbp]
  \centering
  \includegraphics[width=6in]{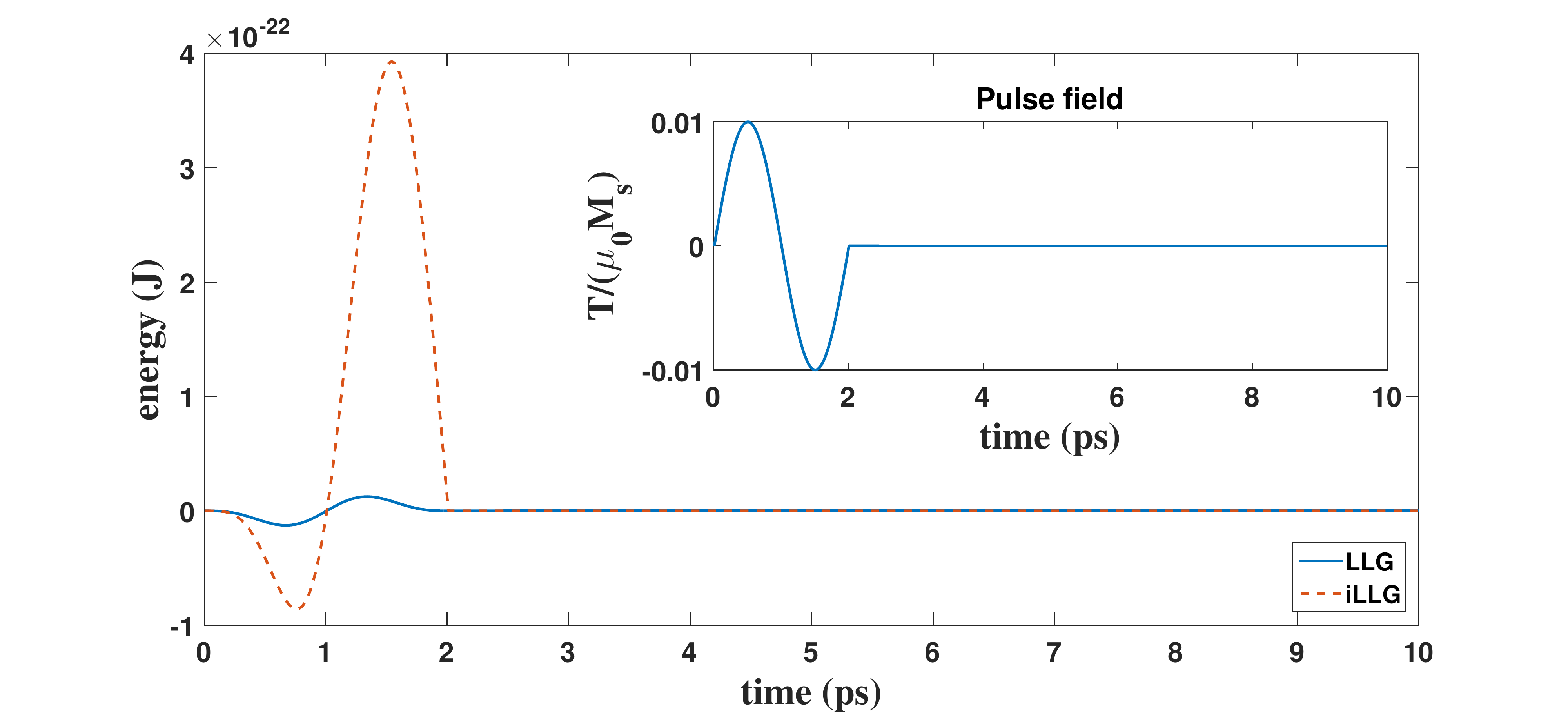}
  \caption{LL energy (without the stray field term) as a function of time when a magnetic pulse is applied for the LLG equation and the iLLG equation, respectively. Insert: profile of the magnetic pulse with frequency $500\mathrm{GHz}$ is applied during $0\leq t\leq 2\mathrm{ps}$.}
  \label{fig:energy_nostray}
\end{figure}

In the presence of stray field, due to its dominance, the ferromagnetic material relaxes to an equilibrium state without any obvious observation of oscillatory magnetization profiles originated from the inertial term over a wide range of $\tau$. Therefore, equilibrium states generated by the iLLG equation are consistent with those generated by the LLG equation; see \cref{fig:equilibrium_inertial}. The ferromagnet is of the size $2\mathrm{\mu m}\times1\mathrm{\mu m}\times0.02\mathrm{\mu m}$ with the meshsize $20\mathrm{nm}\times20\mathrm{nm}\times20\mathrm{nm}$, and $T = 2\mathrm{ns}$ with $\Delta t = 10\mathrm{fs}$. Here $\tau = 1.0\times10^{-10}\mathrm{s}$ and all the other parameters are the same as the Standard Problem \#1~\cite{NIST}.
\begin{figure}[htbp]
  \centering
  \subfloat[Flower state]{\includegraphics[width=3.in]{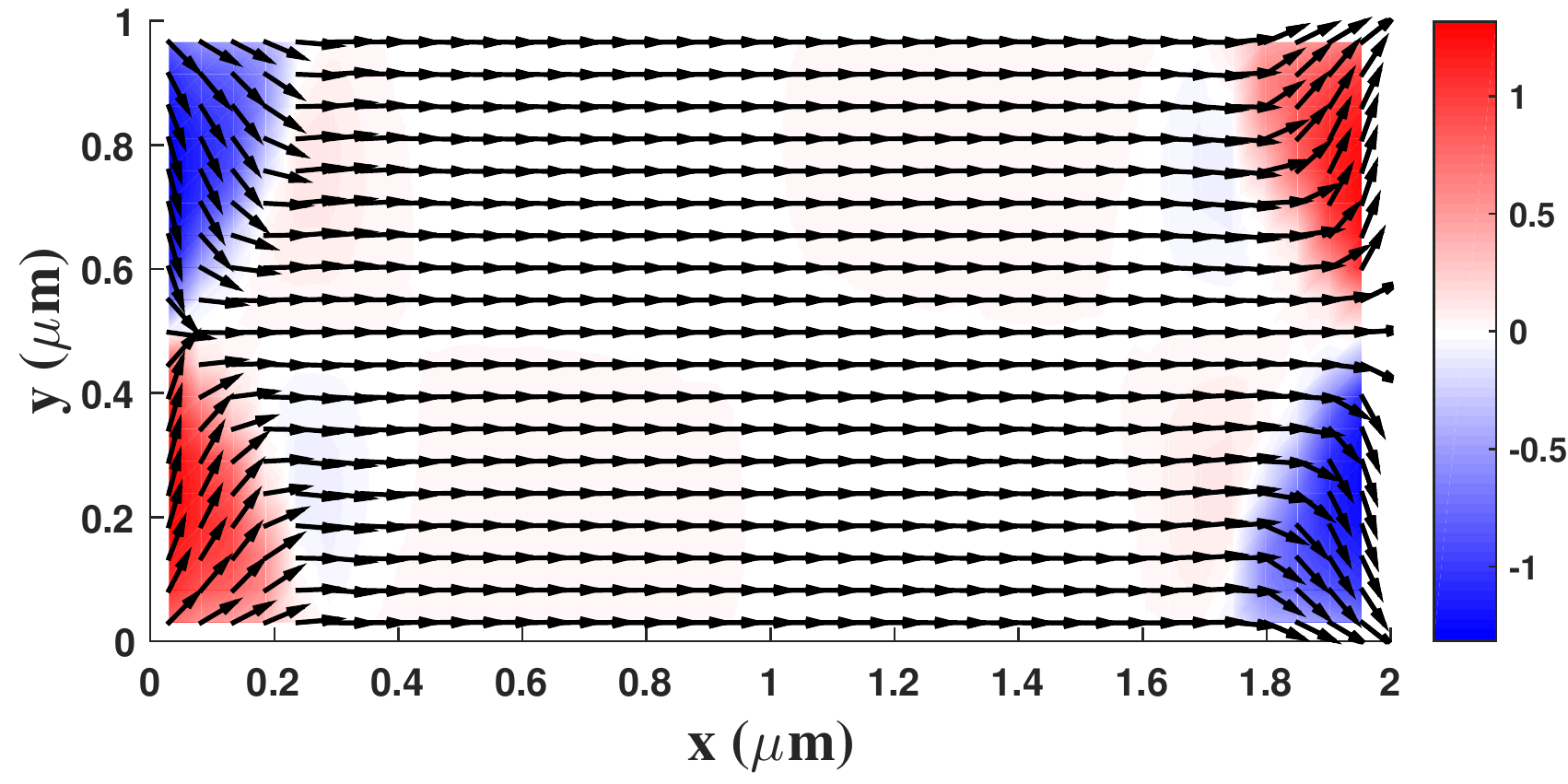}\label{subfig:Flower}}
  \subfloat[C state]{\includegraphics[width=3.in]{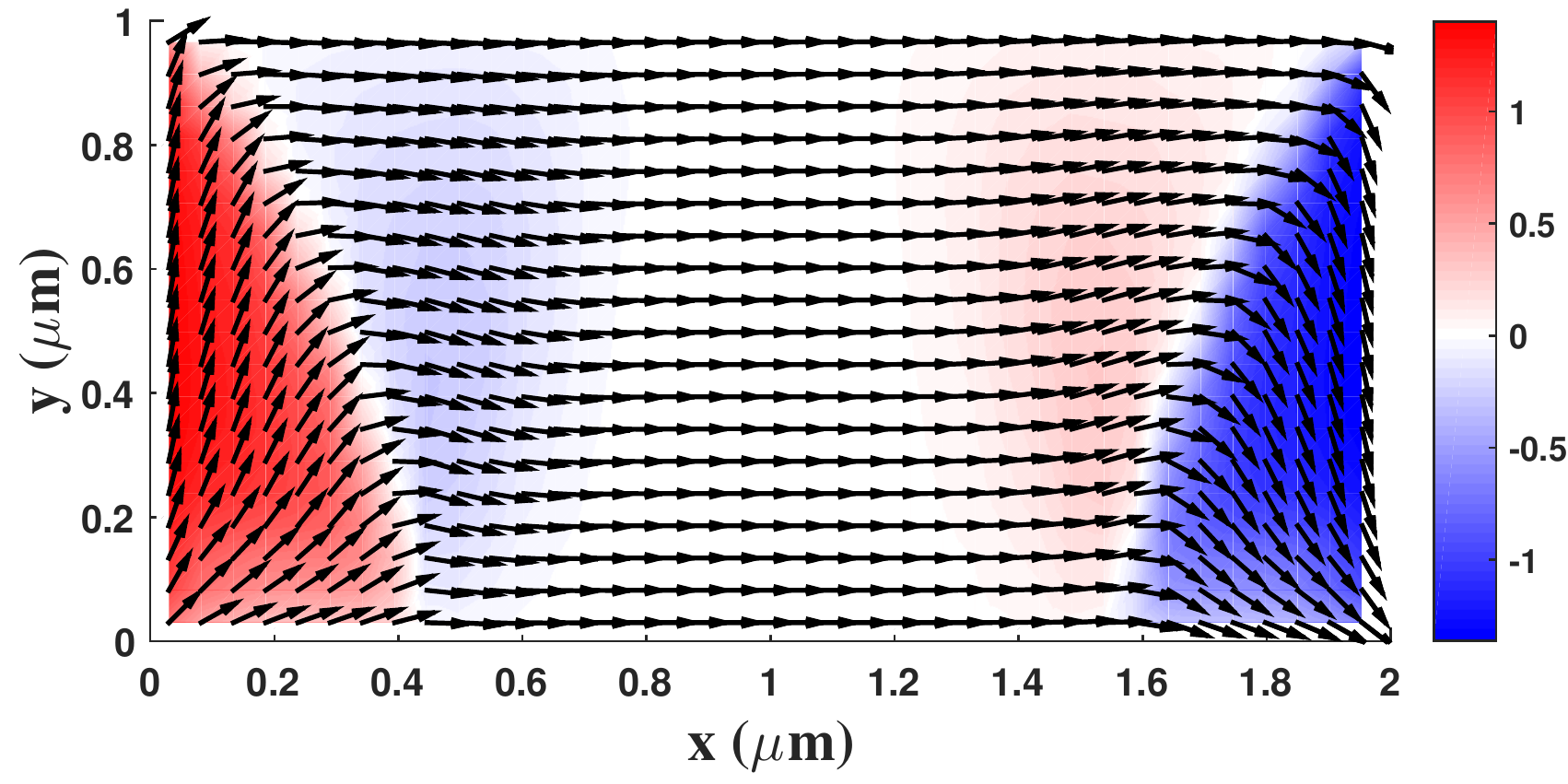}}
  \quad
  \subfloat[Diamond state]{\includegraphics[width=3.in]{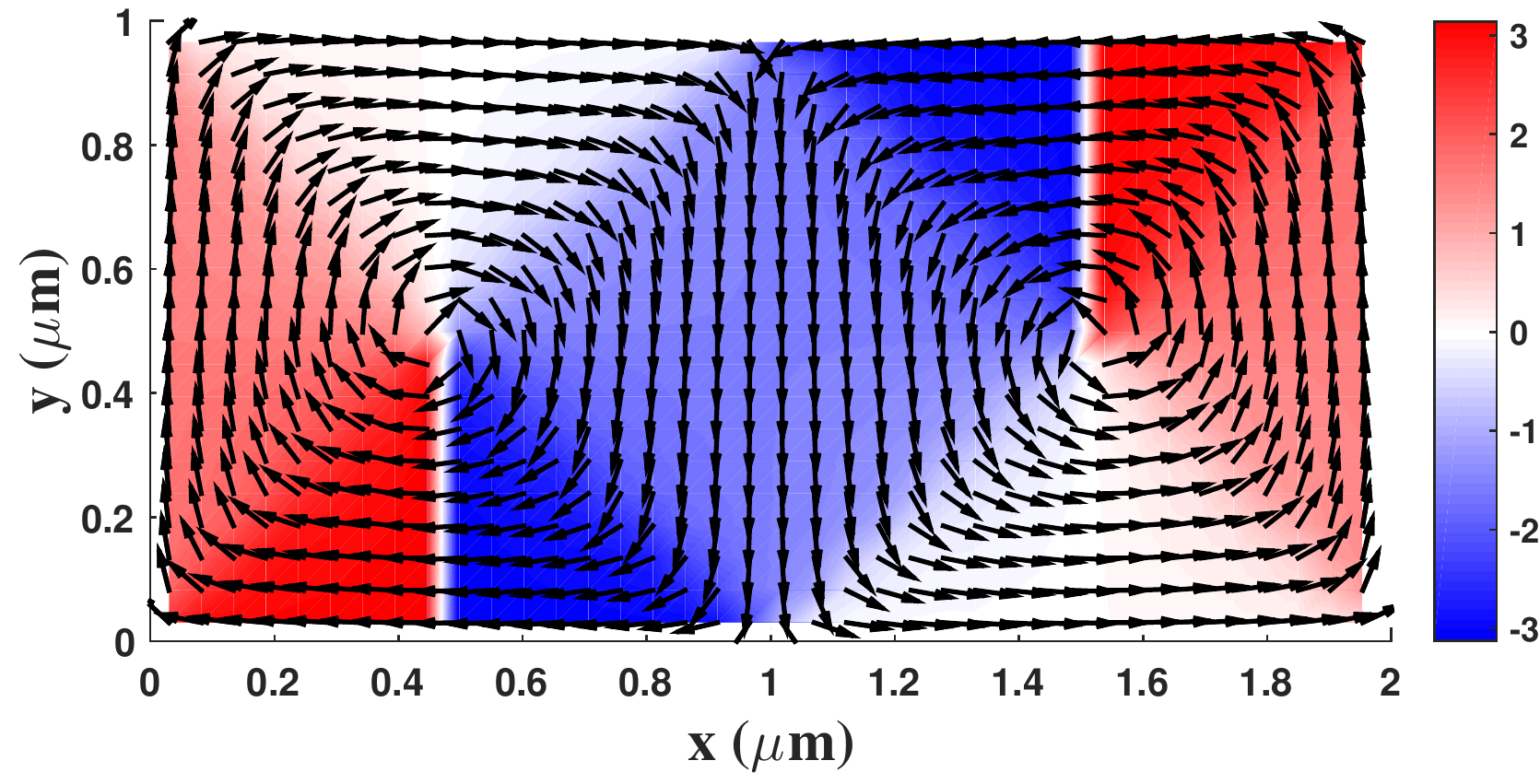}}
  \subfloat[S state]{\includegraphics[width=3.in]{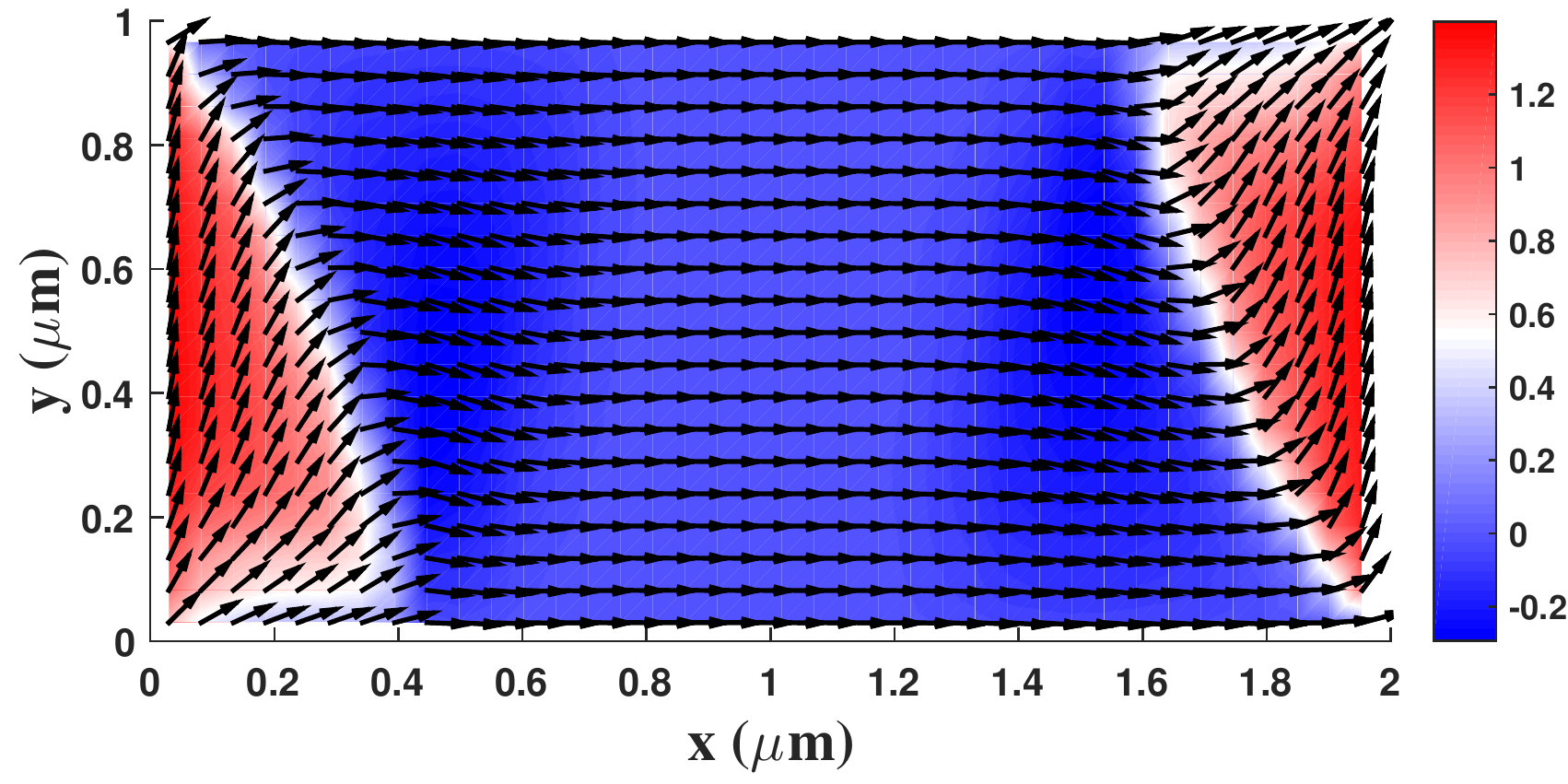}}
  \caption{Magnetization profiles at $2\mathrm{ns}$ with $\Delta t = 10\mathrm{fs}$ when the inertial parameter $\tau = 1.0\times10^{-10}\mathrm{s}$. The arrow represents the in-plane magnetization and the background color represents the angle between the in-plane magnetization and the $x$-axis.}
  \label{fig:equilibrium_inertial}
\end{figure}

Furthermore, following the study of hysteresis loops, we simulate magnetization dynamics using the iLLG equation with two sets of damping parameters and inertial parameters; see \cref{fig:hysteresis-loops}. For comparison, we list the coercive fields and remanent magnetization from NIST (mo96a) and our simulations. When the applied magnetic field is changed, the steady state is reached under the condition that the relative change in LL energy is less than $1.0\times10^{-7}$.
\begin{itemize}
  \item LLG equation (mo96a): coercive field x-loop/y-loop: 2.5/4.9 $\mathrm{mT}$; remanent magnetization $(m_x,m_y,m_z)$: x-loop: (0.15, 0.87, 0.00), y-loop: (-0.15, 0.87, 0.00).
  \item iLLG equation ($\alpha = 0.1$, $\tau = 1.0\times10^{-12}\mathrm{s}$ and $\Delta t = 1\mathrm{ps}$): coercive field x-loop/y-loop: 2.3/5.3 $\mathrm{mT}$; remanent magnetization $(m_x,m_y,m_z)$: x-loop: (0.20, 0.87, 0.11e-04), y-loop: (-0.15, 0.88, 0.14e-04).
  \item iLLG equation ($\alpha = 0.02$, $\tau = 1.0\times10^{-13}\mathrm{s}$ and $\Delta t = 0.1\mathrm{ps}$): coercive field x-loop/y-loop: 2.3/6.3 $\mathrm{mT}$; remanent magnetization $(m_x,m_y,m_z)$: x-loop: (0.20, 0.87, 0.25e-03), y-loop: (-0.14, 0.88, 0.86e-05).
\end{itemize}
The maximum deviation of coercive fields is $1.4 \mathrm{mT}$ and the maximum deviation of the component of remanent magnetization is $0.05$. Qualitative agreements are found for coercive fields and remanent magnetization, which implies the consistency between LLG and iLLG equations in long time simulations.
\begin{figure}[htbp]
  \centering
  \subfloat[LLG equation: mo96a~\cite{NIST}.]{\includegraphics[width=2.8in]{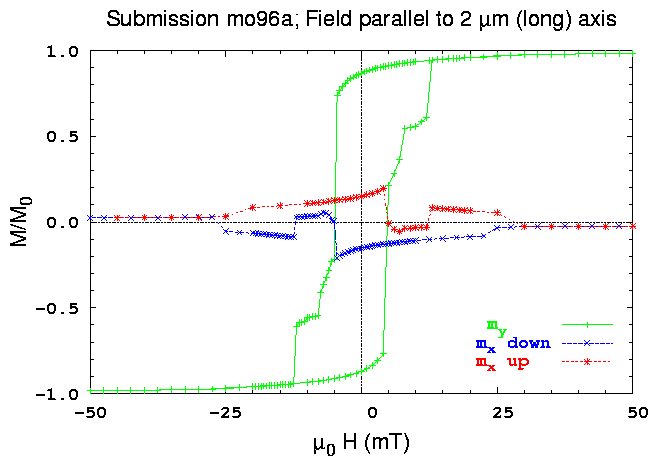}\includegraphics[width=2.8in]{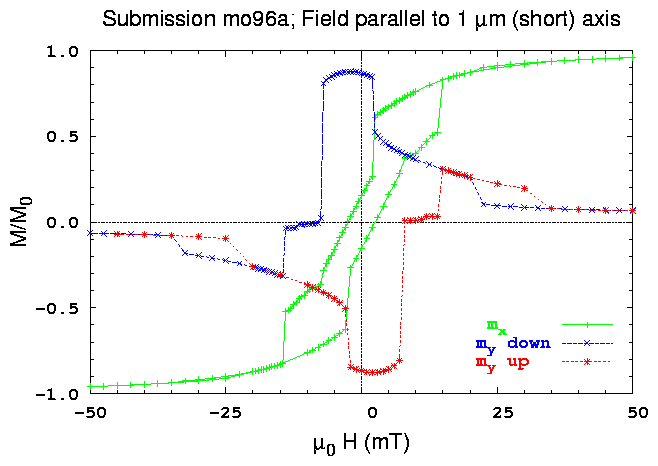}}
  \quad
  \subfloat[iLLG equation: $\alpha = 0.1$, $\tau = 1.0\times10^{-12}\mathrm{s}$ and $\Delta t = 1\mathrm{ps}$.]{\includegraphics[width=2.8in]{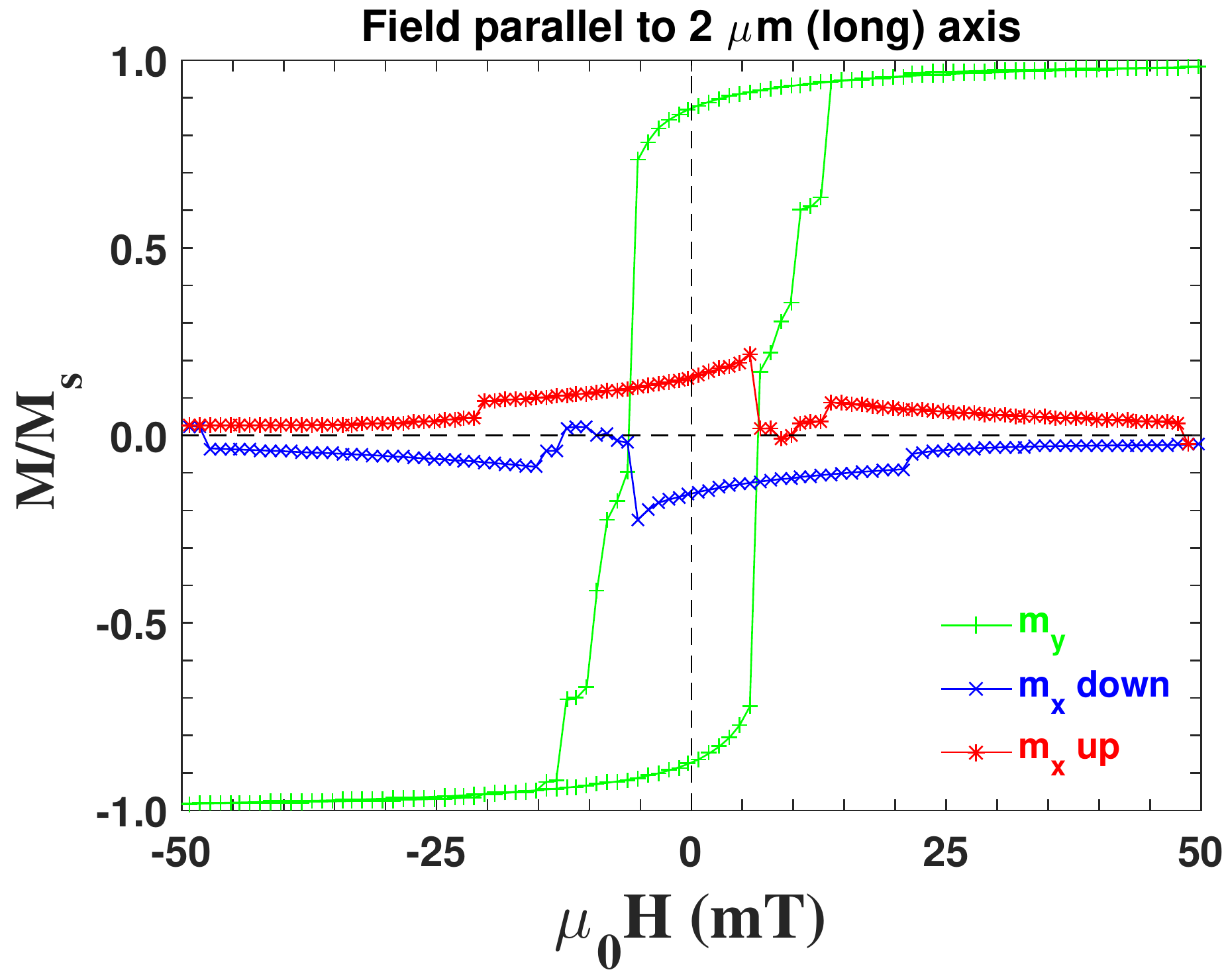}\includegraphics[width=2.8in]{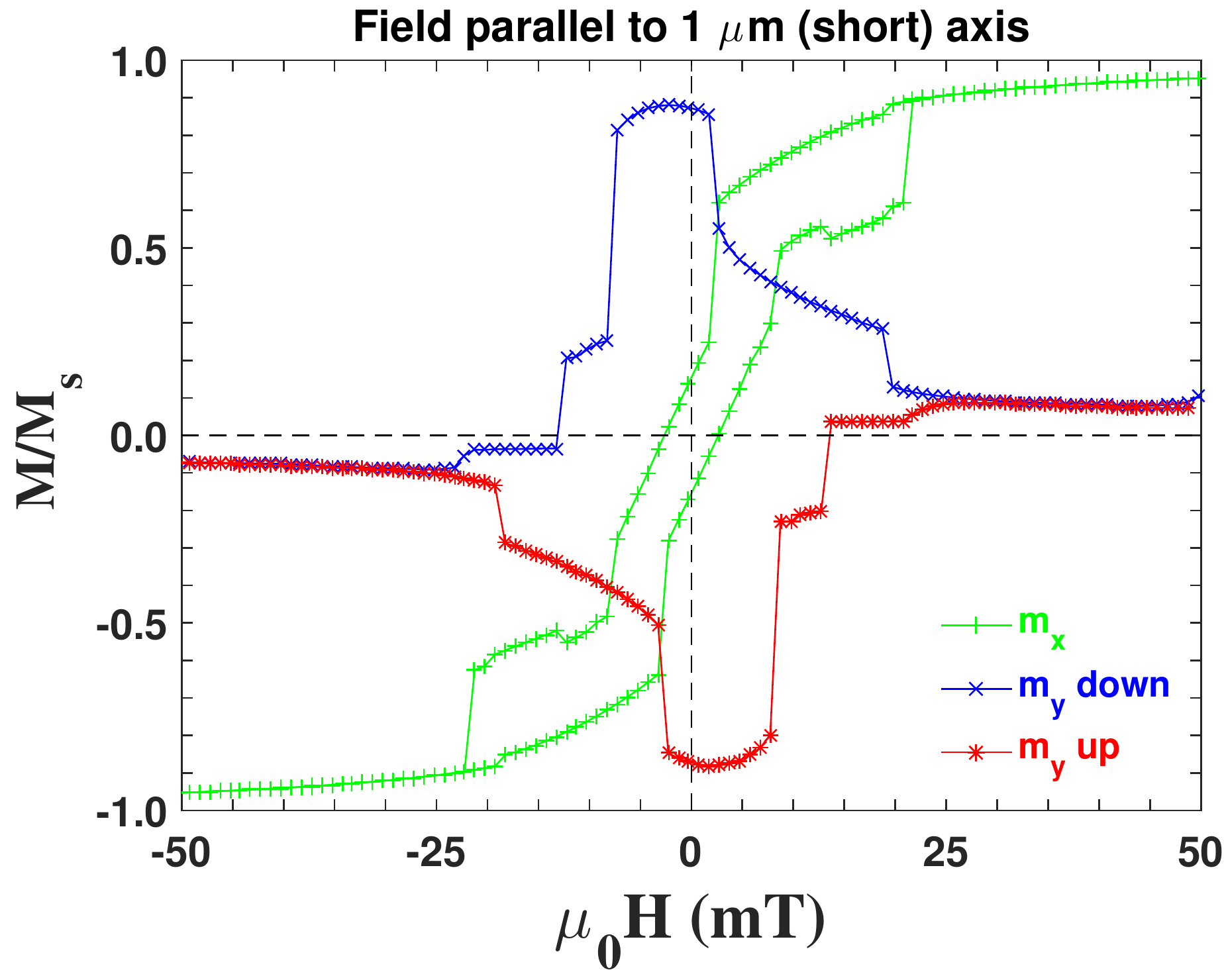}}
  \quad
  \subfloat[iLLG equation: $\alpha = 0.02$, $\tau = 1.0\times10^{-13}\mathrm{s}$ and $\Delta t = 0.1\mathrm{ps}$.]{\includegraphics[width=2.8in]{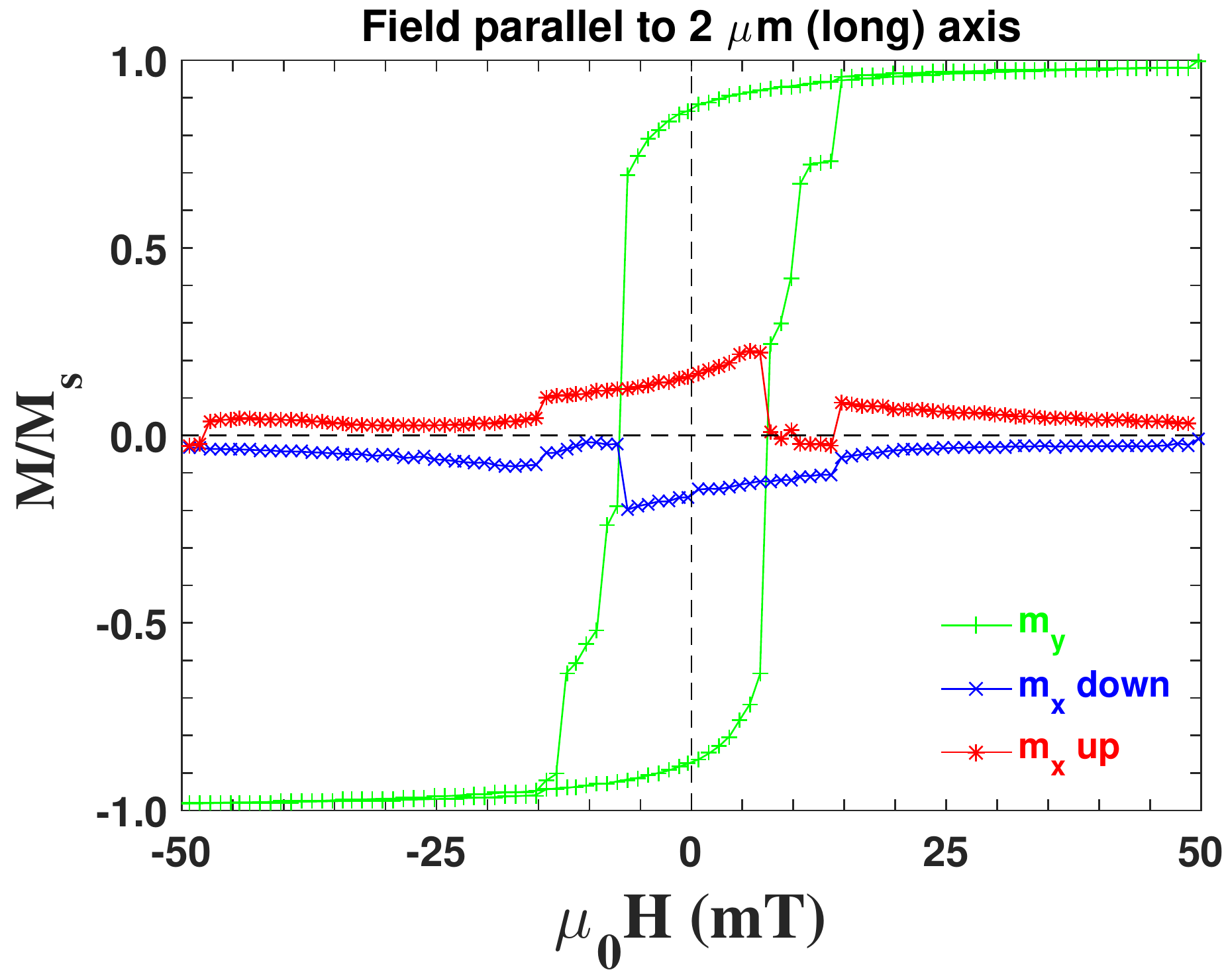}\includegraphics[width=2.8in]{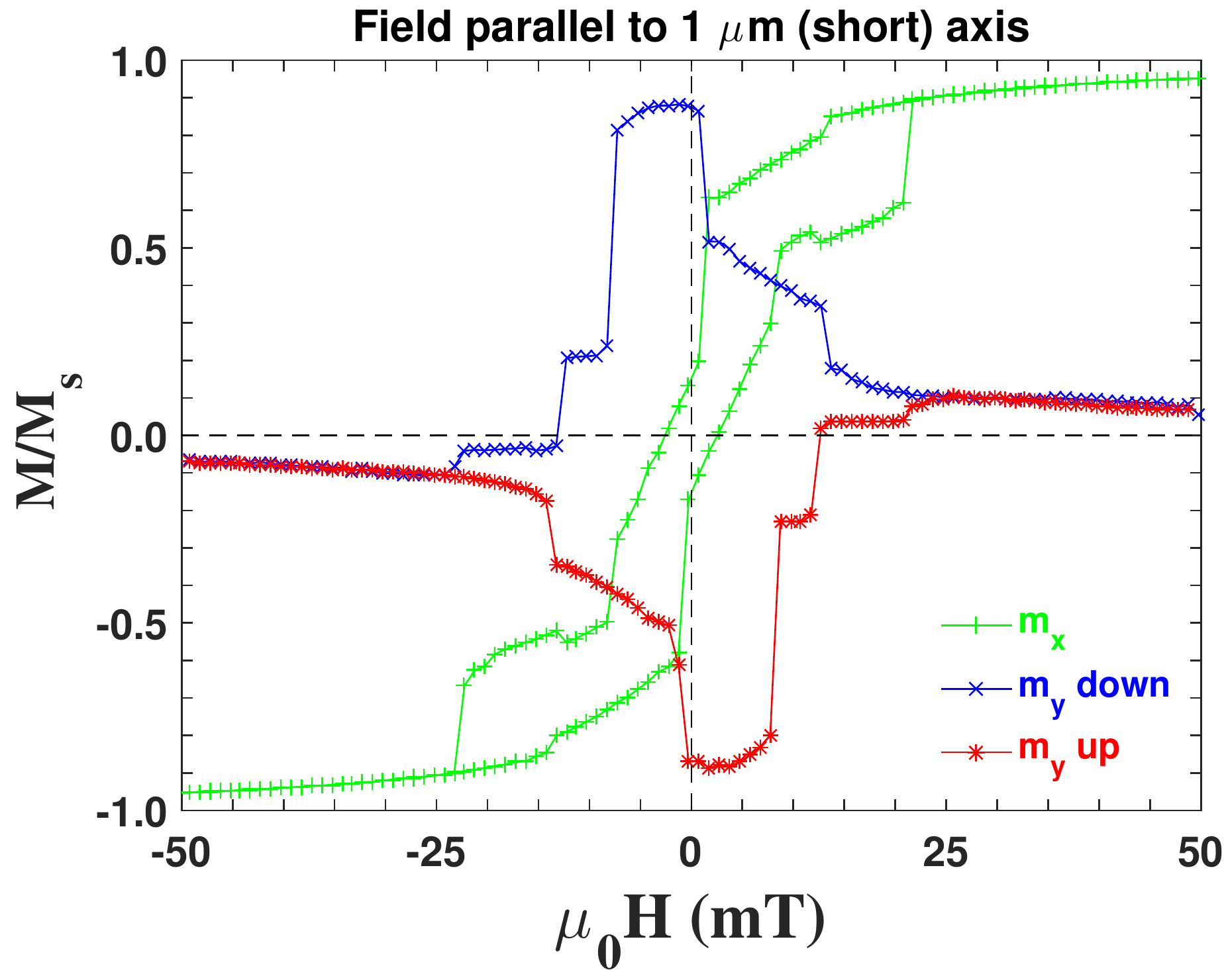}}
  \caption{Hysteresis loops of the LLG equation (mo96a) and the iLLG equation with two sets of damping and inertial parameters. The applied magnetic fields are approximately parallel (canting angle $+1^{\circ}$) to the long axis (left column) and the short axis (right column), and change from $-50\mathrm{mT}$ to $50\mathrm{mT}$ uniformly.}
  \label{fig:hysteresis-loops}
\end{figure}

In the previous simulations, we use a meshsize $20\mathrm{nm}\times20\mathrm{nm}\times20\mathrm{nm}$ which only has one grid point in the $z$ direction. For real 3D simulations, we consider a meshsize $20\mathrm{nm}\times20\mathrm{nm}\times5\mathrm{nm}$. In the absence of an applied magnetic field, it is relaxed to the stable flower state as \cref{subfig:Flower} with the relaxation dynamics shown in \cref{fig:convergence-flower}, and the LL energy $\mathcal{F}[\mathbf{M}]$ and the total energy $\mathcal{J}[\mathbf{M}]$ are plotted as functions of time in \cref{fig:comparsion_FJ} when $\alpha = 0.02$, $\Delta t = 10\mathrm{fs}$, $T = 2\mathrm{ns}$ and $\tau = 1.0\times10^{-11}\mathrm{s}$. At $2\mathrm{ns}$, $\mathcal{F}[\mathrm{M}]$ and $\mathcal{J}[\mathbf{M}]$ are $1.71\times10^{-16}\mathrm{J}$. In the presence of both the stray field and inertial term, the performance of LL energy and total energy are distinct from \cref{fig:energy_nostray}. It is found that the LL energy and the total energy obtained by the proposed method are decaying in the long time but not monotonically on the timescale of sub-picosecond.
\begin{figure}[htbp]
  \centering
  \includegraphics[width=6in]{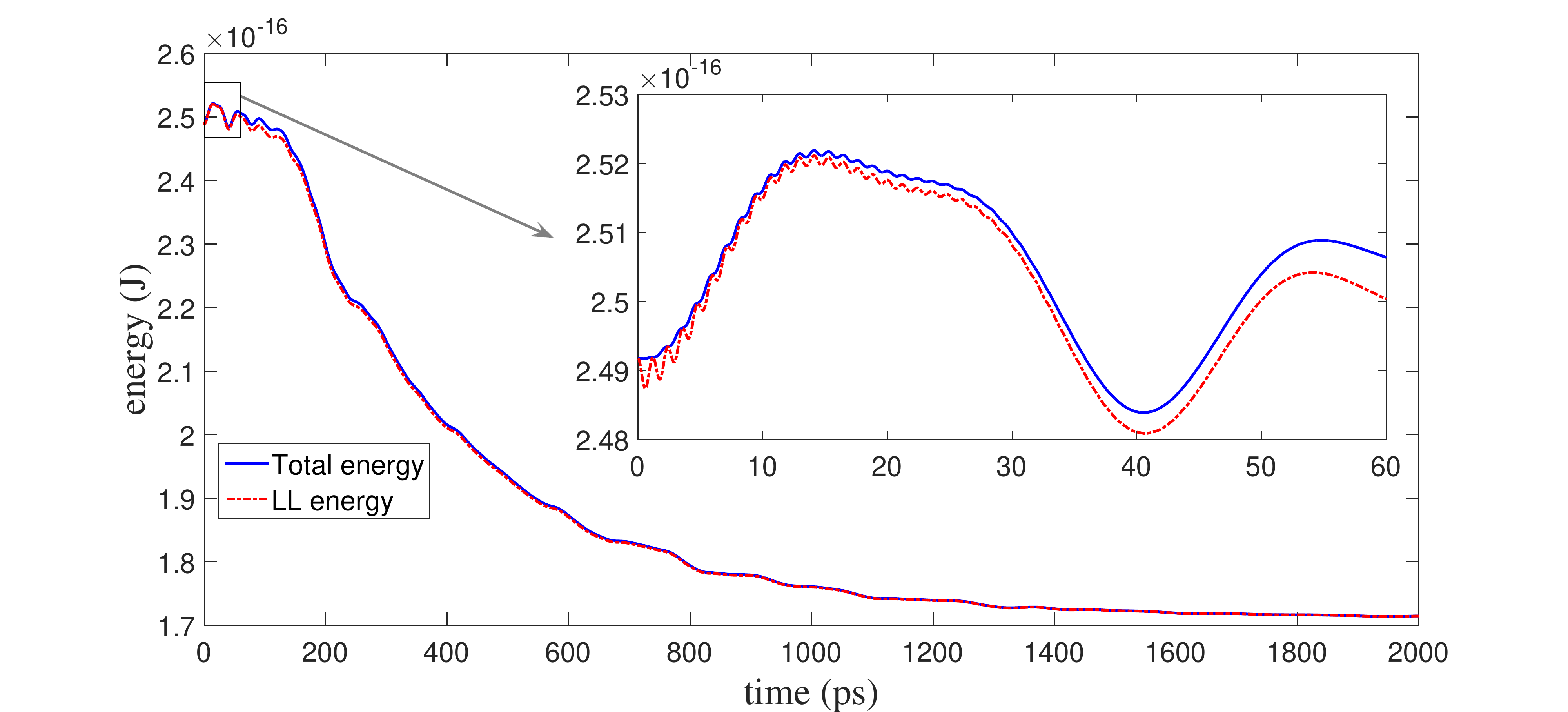}
  \caption{ LL energy $\mathcal{F}[\mathbf{M}]$ and total energy $\mathcal{J}[\mathbf{M}]$ as a function of time in the absence of a magnetic field when $\alpha = 0.02$, $\Delta t = 10\mathrm{fs}$, $T = 2\mathrm{ns}$ and $\tau = 1.0\times10^{-11}~\mathrm{s}$. At $2\mathrm{ns}$, both $\mathcal{F}[\mathrm{M}]$ and $\mathcal{J}[\mathbf{M}]$ are $1.71\times10^{-16}\mathrm{J}$.}
  \label{fig:comparsion_FJ}
\end{figure}

\begin{figure}[htbp]
  \centering
  \subfloat[$t = 0 \mathrm{ps}$]{\includegraphics[width=3in]{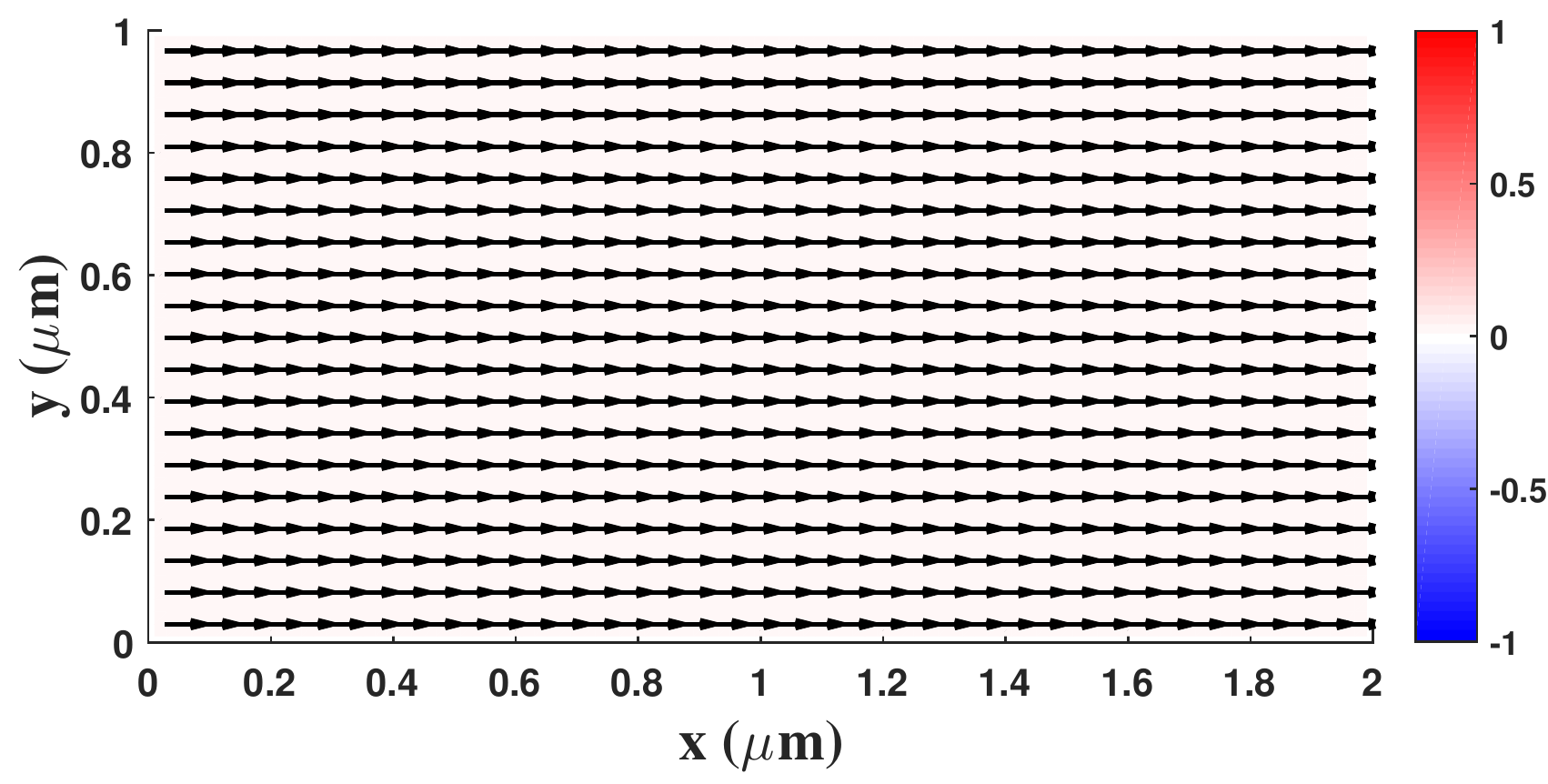}}
  \subfloat[$t = 10 \mathrm{ps}$]{\includegraphics[width=3in]{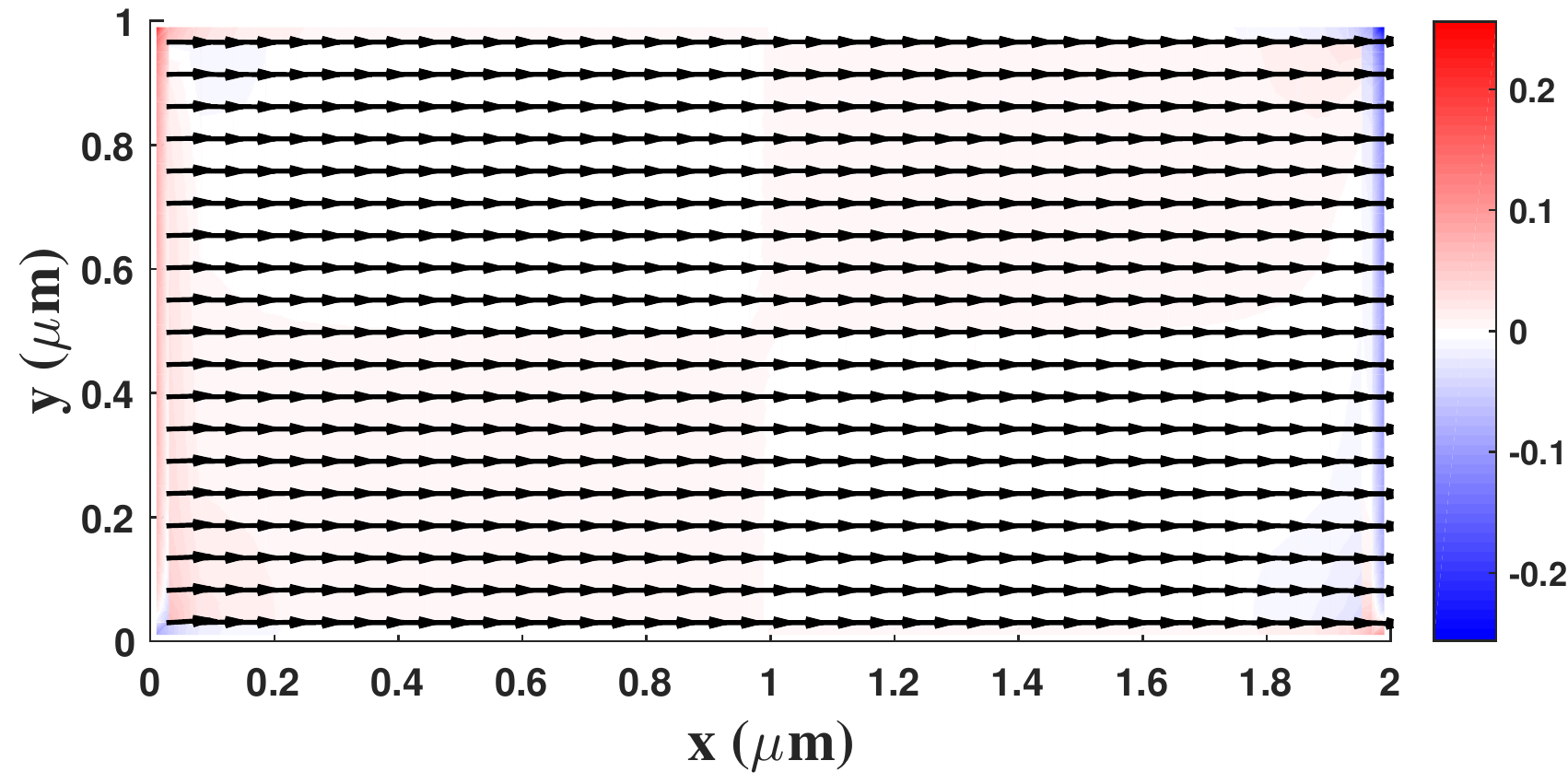}}
  \quad
  \subfloat[$t = 0.1 \mathrm{ns}$]{\includegraphics[width=3in]{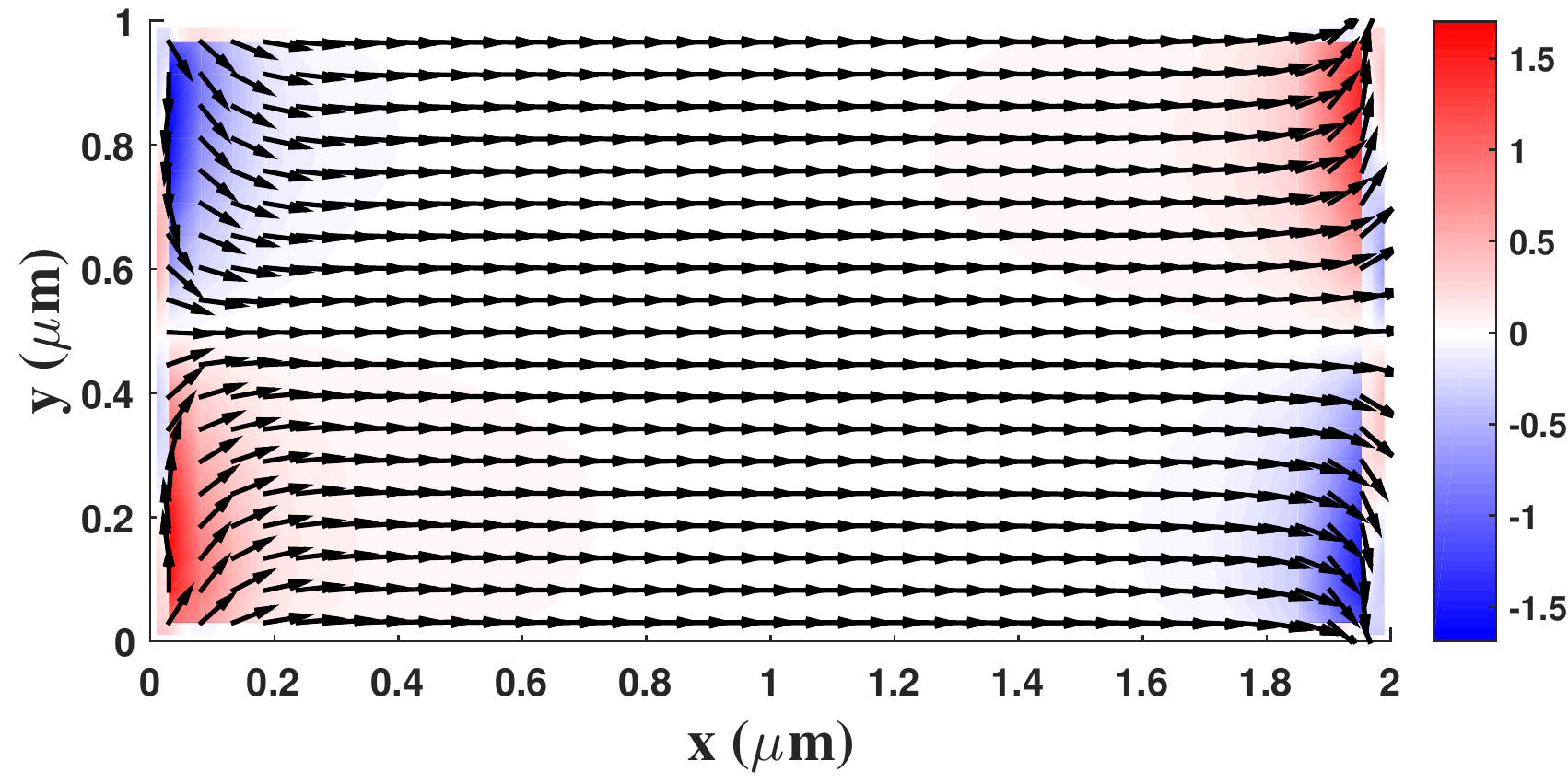}}
  \subfloat[$t = 1 \mathrm{ns}$]{\includegraphics[width=3in]{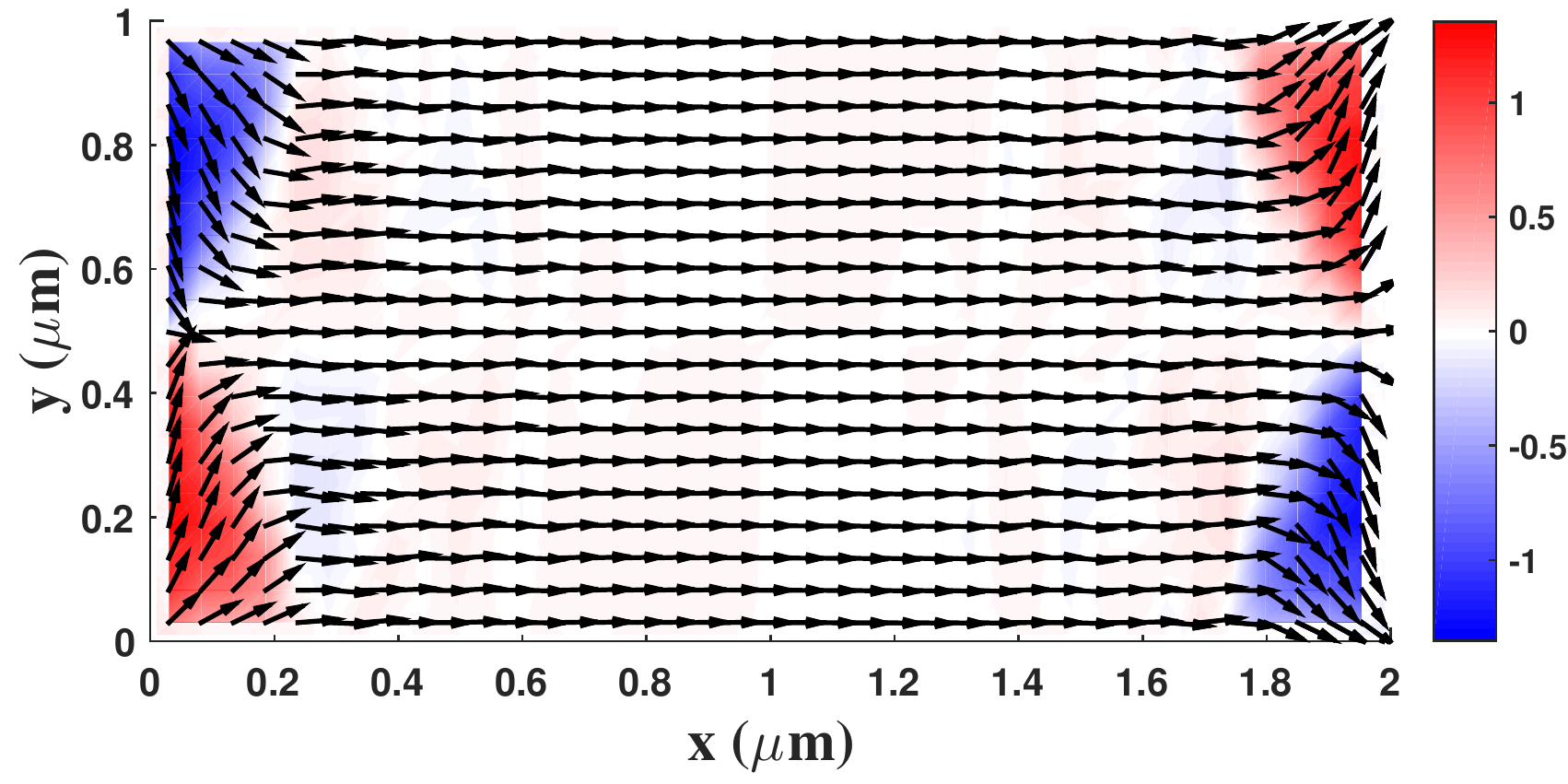}}
  \caption{Magnetization profiles in the $xy$-plane at different times, and the background color represents the angle between the in-plane magnetization and the $x$-axis. Starting from the configuration at $t=0 \mathrm{ps}$, the configuration of magnetization is relaxed to the flower state as in \cref{fig:equilibrium_inertial}.}
  \label{fig:convergence-flower}
\end{figure}

As demonstrated above, the inertial dynamics in the iLLG equation is activated by a high-frequency magnetic pulse. To make a quantitative connection between the inertial dynamics and material parameters, in each case, we only change one material parameter and fix all other parameters. All material parameters are the same as before unless specified.
\begin{itemize}
  \item Change the frequency of magnetic pulse: $f = 200\mathrm{GHz}$, $500\mathrm{GHz}$, and $1000\mathrm{GHz}$. Spatially averaged magnetization and total energy in terms of time for the iLLG equation in the presence of magnetic pulses are plotted in \cref{fig:diff_frequency}. Inertial dynamics are observed and oscillatory profiles exist in the temporal history of total energy. Compared to \cref{fig:energy_nostray}, results in \cref{fig:diff_frequency} show the existence of the inertial dynamics in the presence of stray field. However, with the increase of frequency, the inertial dynamics gradually becomes invisible. Meanwhile, it is interesting that the LL energy $\mathcal{F}[\mathbf{M}]$ has oscillatory profiles whose frequencies seem to be independent of the frequency of the magnetic pulse field, when the inertial effect is induced by the applied magnetic pulse field.
      \begin{figure}[htbp]
        \centering
        \subfloat[Averaged magnetization]{\includegraphics[width = 6in]{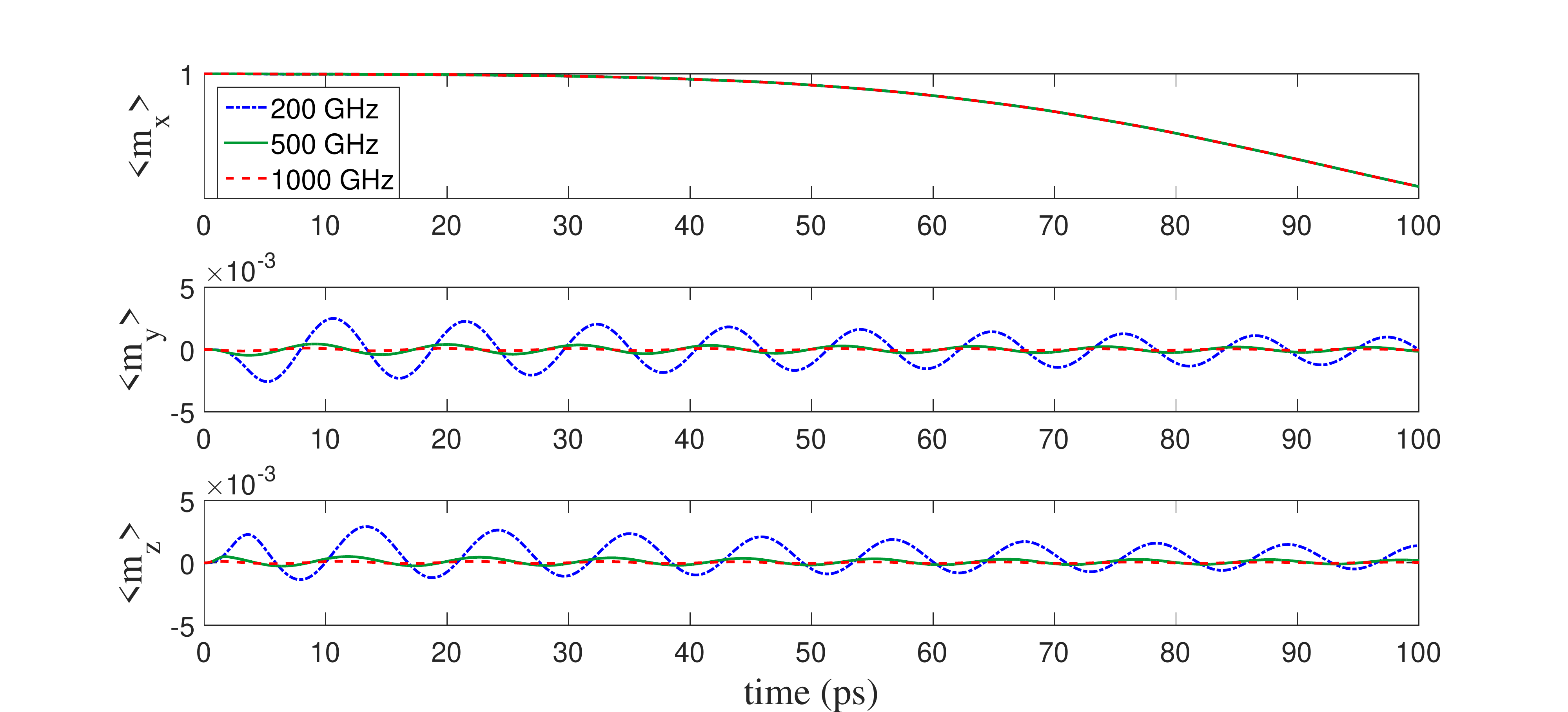}}
        \quad
        \subfloat[LL energy]{\includegraphics[width = 6in]{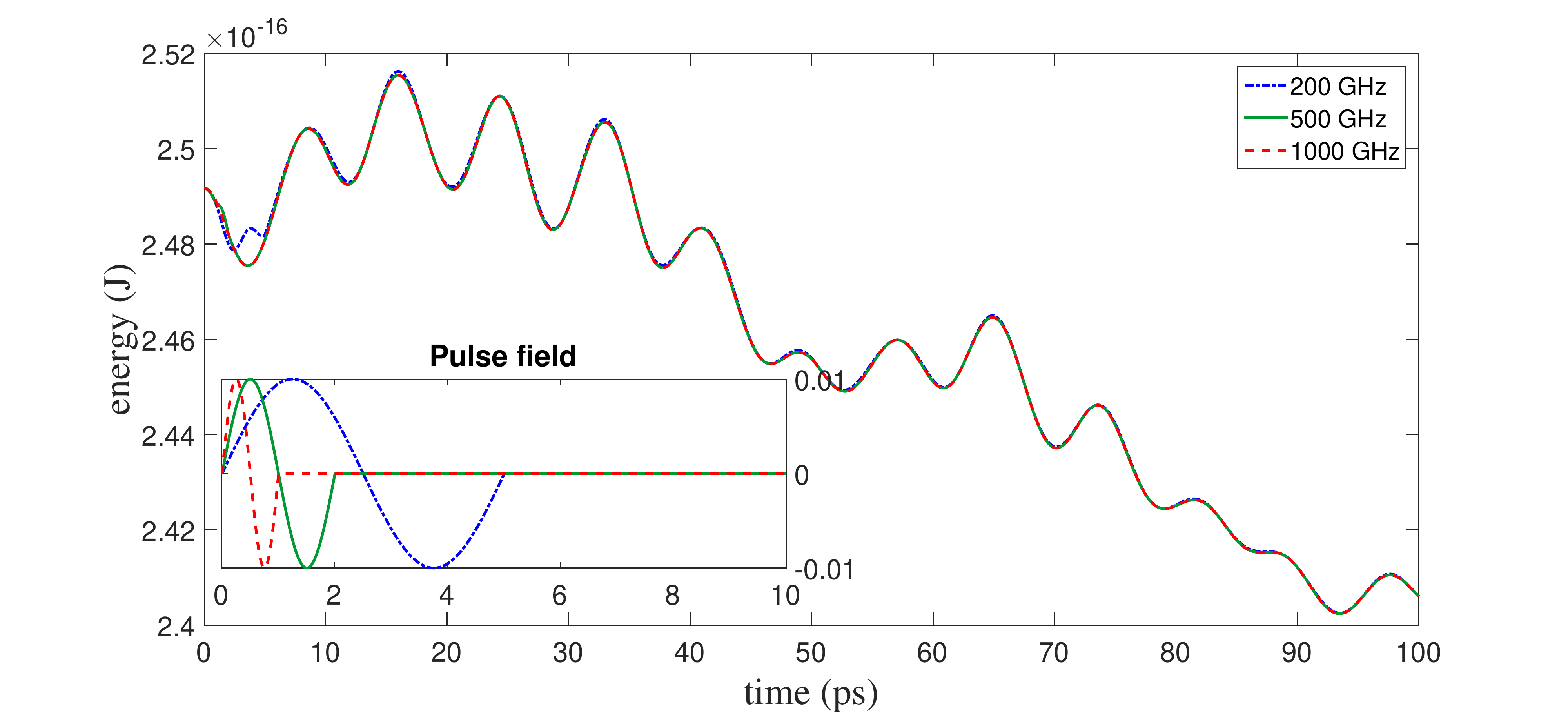}}
        \caption{Averaged magnetization and LL energy as functions of time for the iLLG equation in the presence of magnetic pulse fields with different frequencies ($200\mathrm{GHz}$, $500\mathrm{GHz}$ and $1000\mathrm{GHz}$). Inertial dynamics are observed and oscillatory profiles exist in the temporal history of total energy. Here $\tau = 1.0\times10^{-12}\mathrm{s}$, $\alpha = 0.02$, and $\Delta t = 10\mathrm{fs}$.}
        \label{fig:diff_frequency}
      \end{figure}
  \item Fix the characteristic timescale of the inertial term $\tau = 1.0\times10^{-10}\mathrm{s}$ and change the damping parameter $\alpha = 0.005$, $0.02$, and $0.1$. In the presence of a $500\mathrm{GHz}$ magnetic pulse over $2$ps, we plot averaged magnetization as a function of time for the iLLG equation in \cref{fig:vary damping}. For the parameters under study, the smaller the damping parameter is, the more oscillatory the inertial dynamics is.
      \begin{figure}[htbp]
        \centering
        \includegraphics[width = 6in]{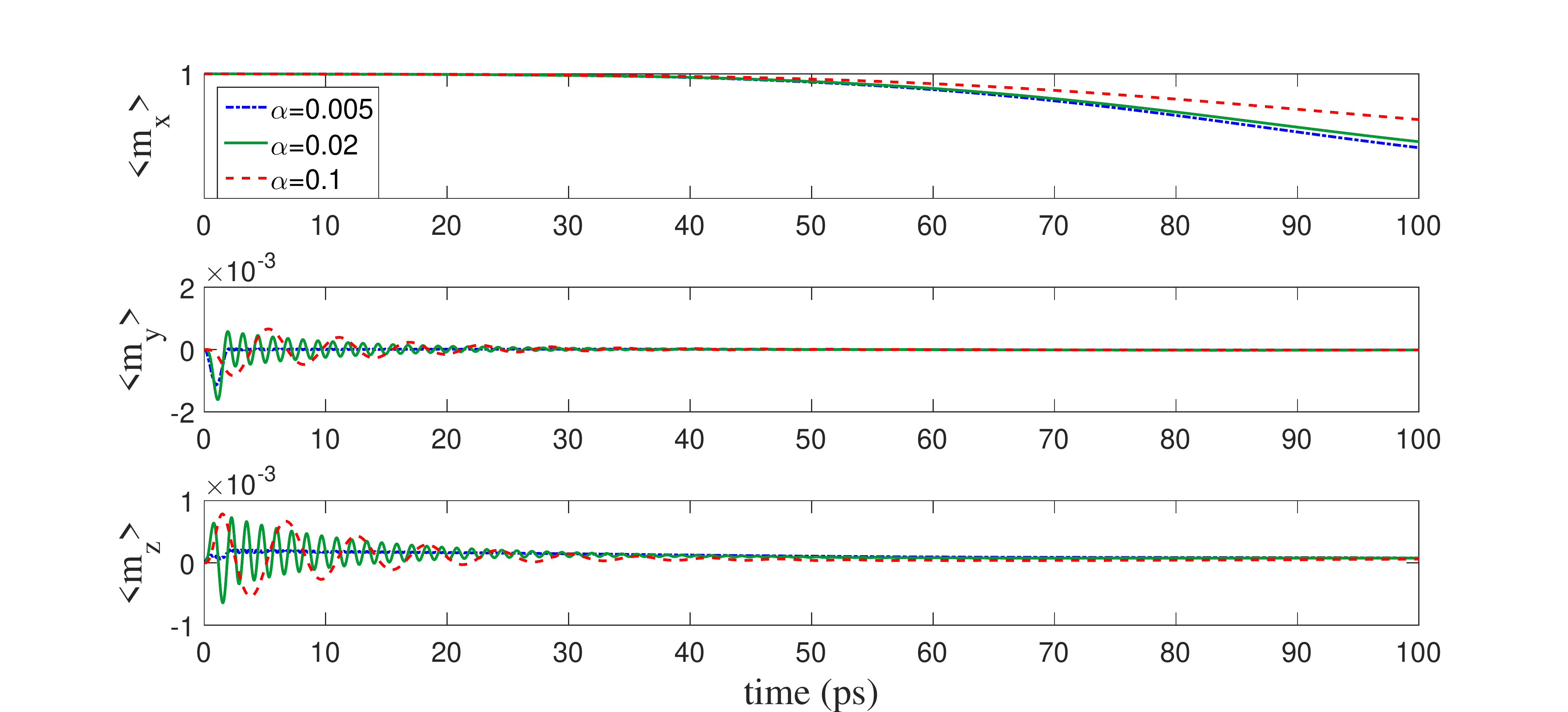}
        \caption{Averaged magnetization in terms of time for the iLLG equation under a $500\;$GHz magnetic pulse when $\alpha = 0.1$, $0.02$ and $0.005$ and $\tau = 1.0\times10^{-10}\mathrm{s}$.}\label{fig:vary damping}
      \end{figure}
  \item Fix the damping parameter $\alpha = 0.02$ and change the characteristic timescale of inertial dynamics $\tau = 1.0\times10^{-9}\mathrm{s}$, $1.0\times10^{-10}\mathrm{s}$ and $1.0\times10^{-11}\mathrm{s}$. We plot the averaged magnetization in terms of time for the iLLG equation under a $500\mathrm{GHz}$ magnetic pulse when $\Delta t = 10\mathrm{fs}$ in \cref{fig:different_tau}. It is found that the smaller the characteristic timescale of inertial term is, the more oscillatory the inertial dynamics is for the material parameters under the study.
      \begin{figure}[htbp]
      \centering
      \includegraphics[width=6in]{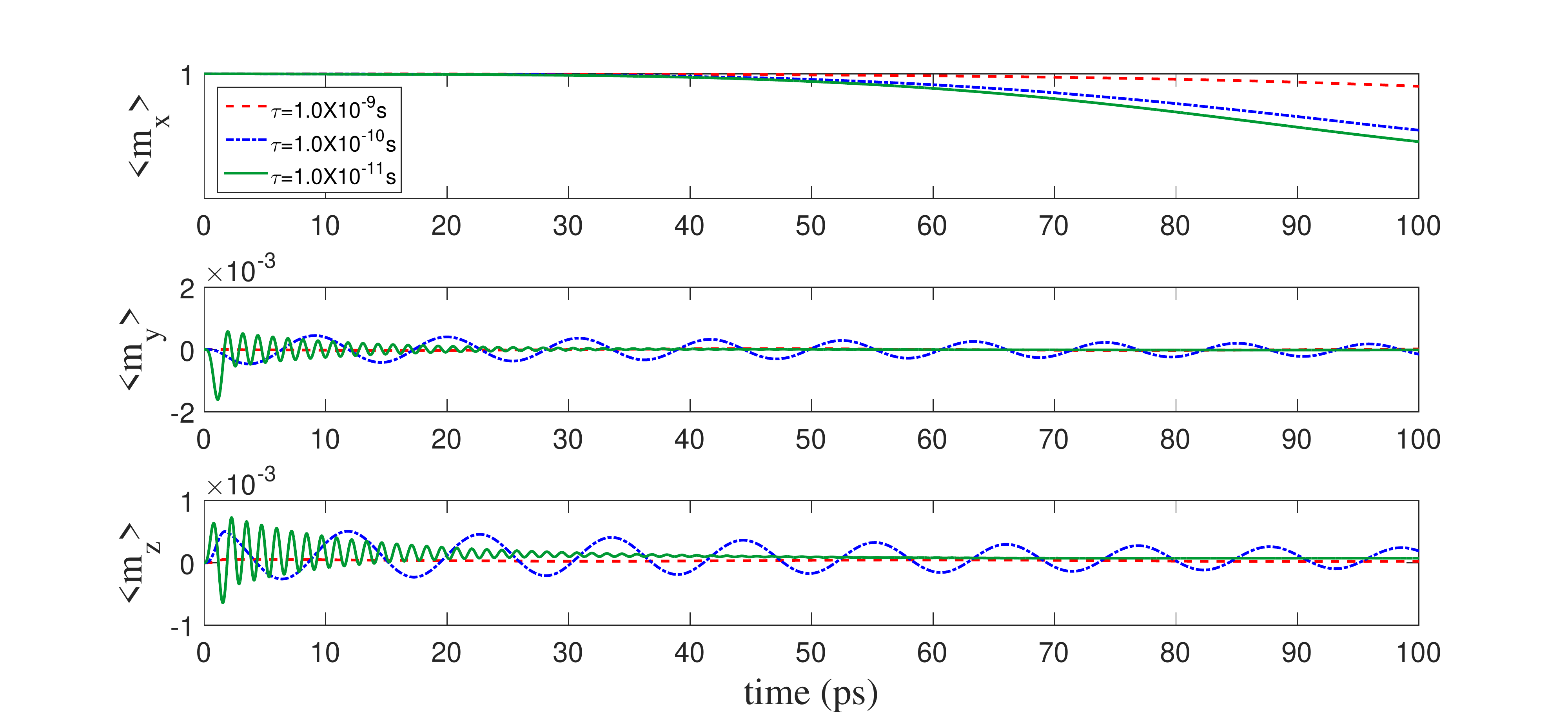}
      \caption{Averaged magnetization as a function of time for the iLLG equation under a $500\mathrm{GHz}$ magnetic pulse when $\tau = 1.0\times10^{-9}\mathrm{s}$, $1.0\times10^{-10}\mathrm{s}$ and $1.0\times10^{-11}\mathrm{s}$, $\alpha = 0.02$, and $\Delta t = 10\mathrm{fs}$.}
      \label{fig:different_tau}
      \end{figure}
\end{itemize}

Finally, we consider an example where the stray field is fully relaxed. An equilibrium state is generated by the iLLG equation with the initial state $\mathbf{m}^{0} = \mathbf{e}_1$, $\alpha = 0.1$, $\tau = 1.0\times10^{-12}\mathrm{s}$, and $\Delta t=1\mathrm{ps}$ for a simulation period $T=2\mathrm{ns}$.
Afterwards, we change $\alpha = 0.005$, $\tau = 5.0\times10^{-11}\mathrm{s}$, and $\Delta t=0.1\mathrm{ps}$, and apply a $500\mathrm{GHz}$ magnetic pulse.
Averaged magnetization, total energy, and LL energy as functions of time are plotted in \cref{fig:stray-relax-generate-pulse}. Magnetization oscillation concentrates around $620\mathrm{GHz}$ when a $500\mathrm{GHz}$ magnetic pulse is applied. The oscillation continues until 250$\mathrm{ps}$ in this simulation.
\begin{figure}[htbp]
  \centering
  \subfloat[Averaged magnetization]{\includegraphics[width = 6in]{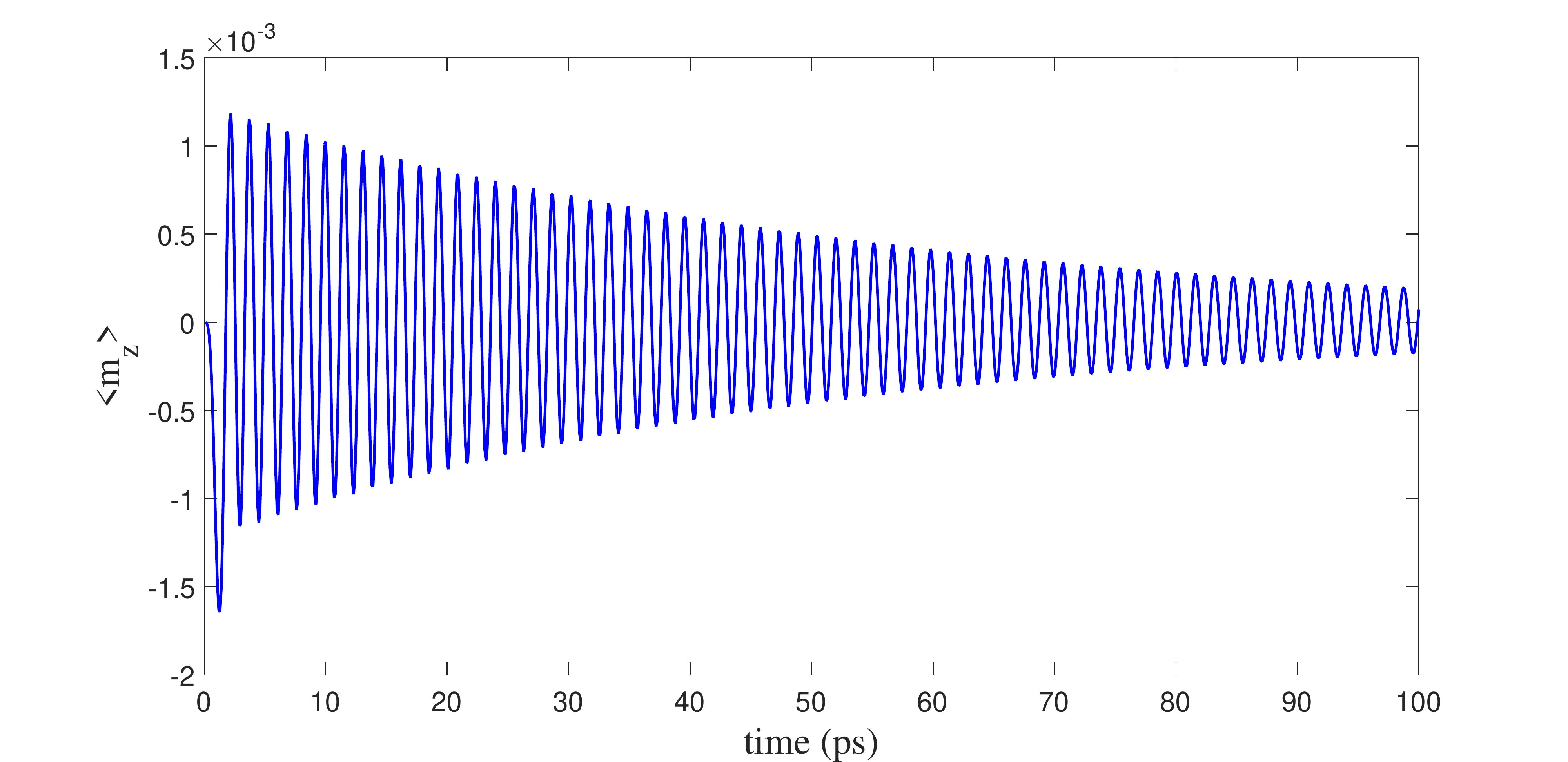}}
  \quad
  \subfloat[Total energy and LL energy]{\includegraphics[width = 6in]{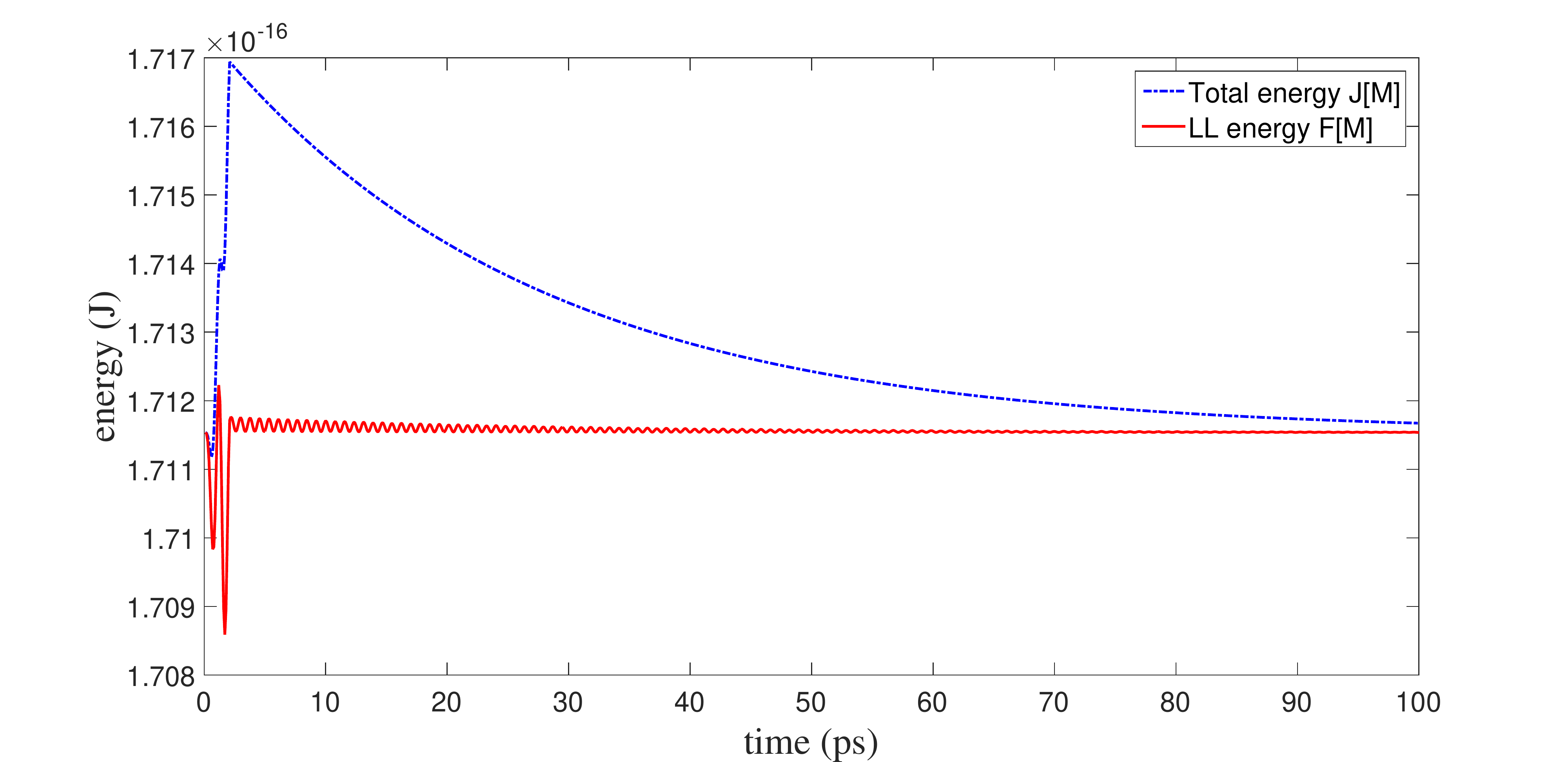}}
  \caption{Averaged magnetization, total energy, and LL energy as functions of time. The initial state is a stable flower state as in \cref{fig:equilibrium_inertial}, which is obtained from the uniform magnetization $\mathbf{e}_1$. Magnetization oscillation concentrates around $620\mathrm{GHz}$ when a $500\mathrm{GHz}$ magnetic pulse is applied. The total energy and LL energy are both $1.71\times10^{-16}\mathrm{J}$ at $200\mathrm{ps}$. The oscillation continues until $250\mathrm{ps}$ in this simulation.}
  \label{fig:stray-relax-generate-pulse}
\end{figure}

\section{Conclusion}
\label{sec:conclusion}
In this work, based on the midpoint scheme, we propose a second-order semi-implicit scheme for the inertial Landau-Lifshitz-Gilbert equation, to study the ultrafast inertial dynamics at the sub-picosecond timescale of ferromagnetic materials. The unique solvability is proven, and the dependence of the number of iterations in GMRES for solving the unsymmetric linear system of equations in the proposed method on the damping/inertial parameters is theoretically explored and further verified by numerical tests. Micromagnetics simulations show the inertial Landau-Lifshitz-Gilbert equation hosts inertial effects at the sub-picosecond timescale, but is consistent with the Landau-Lifshitz-Gilbert equation for larger timescales. Moreover, the relationships between the inertial dynamics and the frequency of the applied field, the damping parameter, and the inertial parameter are systematically investigated. These studies shall be helpful in designing magnetic devices with ultrafast magnetization dynamics of non-negligible inertial behavior.

\section*{Acknowledgments}
P. Li is grateful to Kelong Zheng for helpful discussions, and is also grateful for the discussions with Changjian Xie and Yifei Sun on coding and acknowledges the financial support from the Postgraduate Research \& Practice Innovation Program of Jiangsu Province via grant KYCX20\_2711. L. Yang is supported by the Science and Technology Development Fund, Macau SAR(File no.0070/2019/A2) and National Natural Science Foundation of China (NSFC) (Grant No. 11701598). J. Lan is supported by NSFC (Grant No. 11904260) and Natural Science Foundation of Tianjin (Grant No. 20JCQNJC02020). R. Du was supported by NSFC (Grant No. 11501399). J. Chen is supported by NSFC (Grant No. 11971021).

\bibliographystyle{model1-num-names}
\bibliography{refs}

\end{document}